\documentclass[12pt]{article}

\usepackage{amsfonts}
\usepackage{amssymb}
\usepackage{amsthm}
\usepackage{amsmath}
\usepackage{amscd}
\usepackage{mathrsfs}
\usepackage{enumerate}
\usepackage{ulem}
\usepackage{fontenc}
\usepackage[all]{xy}
\usepackage{array}
\usepackage[english]{babel}

\def\CC {{\mathbb C}}     
\def\NN {{\mathbb N}}     
\def\PP {{\mathbb P}}     
\def\RR {{\mathbb R}}     
\def\XX {{\mathbb X}}     
\def\ZZ {{\mathbb Z}}     

\def\Lw  {\Longrightarrow}
\def\lw  {\longrightarrow}
\def\mc {\mathcal}
\def\mk {\mathfrak}

\def\ol  {\overline}

\def\tst {\Longleftrightarrow}
\def\ul  {\underline}

\def\wt  {\widetilde}
\def\wh  {\widehat}

\newcommand{\scal}[2]{\ensuremath{\langle #1 , #2 \rangle}}

\newtheorem{theorem}{Theorem}[section]
\newtheorem{lemma}[theorem]{Lemma}

\newtheorem{prop}[theorem]{Proposition}

\newtheorem{coro}[theorem]{Corollary}
\newtheorem{rem}{Remark}[section]
\newtheorem{defin}{Definition}[section]

\newtheorem{ex}{Example}[section]

\newtheorem*{thm}{Theorem}
\newtheorem*{corollary}{Corollary}

\begin{document}

\title{On momentum images of representations\\ and secant varieties}

\author{Elitza Hristova, Tomasz Maci\c{a}\.zek, Valdemar V. Tsanov}

\maketitle

\begin{abstract}
Let $K$ be a connected compact semisimple group and $V_\lambda$ be an irreducible unitary representation with highest weight $\lambda$. We study the momentum map $\mu:\PP(V_\lambda)\to\mk k^*$. The intersection $\mu(\PP)^+=\mu(\PP)\cap{\mk t}^+$ of the momentum image with a fixed Weyl chamber is a convex polytope called the momentum polytope of $V_\lambda$. We construct an affine rational polyhedral convex cone $\Upsilon_\lambda$ with vertex $\lambda$, such that $\mu(\PP)^+\subset\Upsilon_\lambda \cap {\mk t}^+$. We show that equality holds for a class of representations, including those with regular highest weight. For those cases, we obtain a complete combinatorial description of the momentum polytope, in terms of $\lambda$. We also present some results on the critical points of $||\mu||^2$. Namely, we consider the existence problem for critical points in the preimages of Kirwan's candidates for critical values. Also, we consider the secant varieties to the unique complex orbit $\XX\subset\PP$, and prove a 
relation between the momentum images of the secant varieties and the degrees of $K$-invariant polynomials on $V_\lambda$.
\end{abstract}

\small{
\tableofcontents
}

\newpage

\section{Introduction}

Let $K$ be a compact connected semisimple Lie group. Let $K\to SU(V)$ be an irreducible unitary representation of a connected compact semisimple group $K$, with a $K$-invariant Hermitian form $\langle.,.\rangle$ on $V$. Let $\mk k$ be the Lie algebra of $K$. Then
$$
\mu = \mu_K : \PP(V) \lw {\mk k}^* \;,\quad \mu[v](\xi)=\frac{1}{i}\frac{\langle\xi v,v\rangle}{\langle v,v\rangle} \;,\; [v]\in\PP,\xi\in{\mk k}
$$
is a $K$-equivariant momentum map on $\PP(V)$ with respect to the Fubini-Study form on $\PP(V)$. This map, and specifically its image $\mu(\PP(V))$, are the subject of these notes.

Momentum maps of Hamiltonian $K$-manifolds, or algebraic varieties, have a rich theory. The projective space of a unitary representation is a prominent example, which also has considerable bearing on the general theory. Remakably, to the best of our knowledge, the momentum images of irreducible representations are still not known, in the sense that there is no concrete closed description of the image $\mu(\PP(V))$ in terms of the invariants of the representation, e.g. the highest weight. This is one of the questions we discuss in this text. We show that for an irreducible representation $V$ with highest weight $\lambda$ the intersection of the momentum image $\mu(\PP(V))$ with a positive Weyl chamber is always contained in the intersection of two rational polyhedral cones. Moreover, for a large class of representations, including the case when $\lambda$ is regular, the image is in fact equal to this intersection. Our method is based on Wildberger's method, \cite{Wildberger}, which he developed in order to 
determine all irreducible representations with convex $\mu(\PP(V))$. In his paper \cite{Wildberger} Wildberger defines the notion of root-distinct weights, which plays a key role in our work. We believe that a further refinement of our method could lead to a complete description of the momentum image also in the cases when the image is strictly contained in the intersection mentioned above.

A systematic study of momentum images of representations was also initiated recently by Vergne and Walter, \cite{Vergne-Walter-MomCone}. They do not assume irreducibility of the representation. They obtain a general qualitative description of the momentum image, under the assumption that the generic isotropy group is finite. The specification of this description to a given case remains as a nontrivial task. In their proof, Vergne and Walter use semistability methods, related to previous work of Ressayre, among others. Our methods are independent, based on momentum map techniques and the combinatorics of roots and weights. Certainly, the two descriptions should yield the same result and we believe that the interaction of the two methods could bring new insights into the structure of momentum images.

To state more precisely our results, let us introduce some notations and recall relevant known facts. Let $T \subset K$ be a Cartan subgroup with corresponding Lie algebra ${\mk t}$ and Weyl group $W$. Let $B\subset K$ be a Borel subgroup with $T\subset B$. We use the identification ${\mk t} \cong {\mk t}^*$ and denote by ${\mk t}^+$ the positive Weyl chamber corresponding to the choice of $B$. Let $\Pi$ be the set of simple roots. Let $\lambda$ be a dominant weight and $\PP=\PP(V_\lambda)$. For $\alpha\in\Pi$ we denote by $\omega_\alpha$ the corresponding fundamental weight and write
$$
\lambda = \sum\limits_{\alpha\in\Pi} \lambda_\alpha\omega_\alpha \;.
$$
We partition the set of simple roots in the following way:
\begin{gather*}
\begin{array}{lll}
\Pi &= \Pi_{\rm rd} \sqcup \Pi_1 \sqcup \Pi_0 \;, \\
\hline
\Pi_{\rm rd} &= \{\alpha\in\Pi: \lambda_\alpha\geq 2\} \\
\Pi_{1} &= \{\alpha\in\Pi: \lambda_\alpha=1\} \\
\Pi_{0} &= \{\alpha\in\Pi: \lambda_\alpha=0\} \\
\end{array}
\end{gather*}
and we further denote $\Pi^{\lambda} = \Pi_{\rm rd} \sqcup \Pi_1$. Let $\chi_0\in\mk t^+$ be the element determined by $(\chi_0|\alpha)=||\chi_0||^2$ for $\alpha\in\Pi^\lambda$ and $(\chi_0|\beta)=0$ for $\beta\in\Pi_0$.

It is well known that the intersection $\mu(\PP)^+=\mu(\PP)\cap\mk t^+$ is a convex polytope, called the momentum polytope of $\PP$. This polytope intersects every $K$ orbit in $\mu(\PP)$ at exactly one point, and thus determines the momentum image. The highest weight is the image of the highest weight vector, $\lambda=\mu[v^\lambda]$. The function $||\mu||^2$ has a unique maximum $||\lambda||^2$, attained exactly at the orbit $K[v^\lambda]$, which is the unique complex orbit of $K$ in $\PP$. The intersection $\mu(\PP)^{\mk t}=\mu(\PP)\cap\mk t^*$ is preserved by $W$. We have $\mu(\PP)^{\mk t}\subset\mc C(W\lambda)$ and $\mu(\PP)\subset\mc C(K\lambda)$, where $\mc C(-)$ denotes convex hull in $\mk k^*$. In particular, the momentum image is convex if and only if $\mu(\PP)=\mc C(K\lambda)$. We consider affine convex cones with vertex $\lambda$, and aim to obtain the momentum polytope as an intersection of such a cone with the Weyl chamber. We start with the cone ${\rm Cone}_\lambda(W\lambda)$ generate by the 
rays $\lambda+\RR_+(w\lambda-\lambda)$. Clearly, $\mu(\PP)^{\mk t}\subset{\rm Cone}_\lambda(W\lambda)$. Wildberger has shown that $\mu(\PP)$ is convex if and only if $\Pi_1=\emptyset$. In particular, for $\alpha\in\Pi_1$, the edge $E^\lambda_{s_\alpha\lambda}$ is not contained in $\mu(\PP)$. We define a subcone of ${\rm Cone}_\lambda(W\lambda)$ by introducing an inequality for each such edge. Let $W_\lambda$ be the stabilizer of $\lambda$ in $W$. We define
$$
\Upsilon_\lambda = \{\nu\in{\rm Cone}_\lambda(W\lambda): (\nu|ws_\alpha \chi_0)\leq(\lambda|ws_\alpha\chi_0) , \forall \alpha\in\Pi_1,\forall w\in W_\lambda\} \;.
$$
In Sections \ref{The case of a regular highest weight} and \ref{A cone containing the image} we prove the following results:

\begin{thm}
For every dominant weight $\lambda$, we have
$$
\mu(\PP(V_\lambda)) \cap \mk t^+\subset\Upsilon_\lambda \cap \mk t^+ \; .
$$
Furthermore, when $\lambda$ is regular we have an equality
$$
\mu(\PP(V_\lambda)) \cap \mk t^+ = \Upsilon_\lambda \cap \mk t^+ \; .
$$
\end{thm}
The regular case is not the only case when equality is achieved, our most general result is stated in Section \ref{The intersection of two cones}. Equality also holds for all known to us examples in case $K$ is simply laced. Let us note, however, that equality does not always hold, as shown in example \ref{Ex Upsilon for Sp4}.

The cone $\Upsilon_\lambda$ is $W_\lambda$-invariant by construction. Hence we have
$$
W_\lambda\mu(\PP)^+ \subset \Upsilon_\lambda \;.
$$
On the other hand $W_\lambda(\Upsilon_\lambda\cap{\mk t}^+) = \Upsilon_\lambda\cap(W_\lambda\mk t^+)$. The latter is convex, being the intersection of two convex cones. Thus we get the following.

\begin{corollary}
If $\mu(\PP)^+ = \Upsilon_\lambda\cap\mk t^+$ holds, then the set $\mu(\PP)\cap(W_\lambda\mk t^+)$ is convex.
\end{corollary}

Besides the momentum image itself, we seek to understand further properties of the momentum map for the specific case in hand. We study the critical points of $||\mu||^2$, which is an essential step toward the construction of the Kirwan-Ness stratification of $\PP(V)$ with respect to the $K$-action, \cite{Kirwan}, \cite{Ness-StratNullcone}. Let us call $\xi\in\mk k^*$ and intermediate critical value, if $\mu^{-1}(\xi)$ contains critical points of $||\mu||^2$. In \cite{Kirwan} Kirwan gives a necessary condition for an element $\xi \in \mk t$ to be an intermediate critical value for $||\mu||^2$. It is, however, not true that this condition is always sufficient, in some cases the critical set turns out to be empty. In Section \ref{Critical points} we show that when $\xi$ is a regular element of $\mk t$, then Kirwan's condition is also sufficient. For the case of singular $\xi$ we describe a procedure on how to determine when $\xi$ is a critical value. Then, in Section \ref{Example qubits} we illustrate our 
procedure on a specific example, namely the representation of $GL_2^{\times N}$ on $(\CC^2)^{\otimes N}$.

The third topic of these notes are secant varieties to homogeneous projective varieties. We seek to relate the properties of the momentum image to the projective geometry of $K$-orbits in $\PP(V)$. It is well-known that $V$ is irreducible if and only if $K$ has exactly one complex orbit $\XX\subset\PP(V)$, the orbit of a highest weight vector, a flag variety. To any irreducible nondegenerate projective variety, $\XX$ in our case, is associated a sequence of nested irreducible subvarieties
$$
\XX = \Sigma_1 \subset\Sigma_2 \subset\dots\subset\Sigma_{r_m}=\PP(V)\;,
$$
called the secant varieties of $\XX$ in $\PP$, defined as follow. The projective linear span of any $r$ points on $\XX$ is called an $r$-secant to $\XX$. The secant variety $\Sigma_r$ is Zariski closure of the union of all $r$-secants. The variety $\XX\subset\PP$ is called rank-semi-continuous if the union of $r$-secants is Zariski closed for all $r$. The secant varieties are preserved by $K$. We are interested in the general question: how do the secant varieties intersect with the Kirwan strata? What we present here are some first steps concerning the semistable stratum and degrees of $K$-invariant polynomials on $V$. Let $J\subset\CC[V]^K$ denote the ring of invariant polynomials vanishing at $0$. Recall that 
\begin{center}
$J\ne 0$ if and only if $0\in\mu(\PP)$. 
\end{center}
Assume that $J\ne0$ and consider the minimal positive degree of invariants:
$$
d_1=\min\{d\in\NN:J_d=\CC[V]_d^K\ne 0\}\;. 
$$
We define the rank of semistability for the representation $(K,V)$ as
$$
r_{ss} = \min\{r\in\NN:0\in\mu(\Sigma_r)\} \;.
$$
In section \ref{Sect SecVar} we prove the following.

\begin{thm}
The inequality $r_{ss}\leq d_1$ holds. If furthermore $\XX\subset\PP$ is rank-semi-continuous and $K$ is the full connected isometry group of $\XX$, then $r_{ss}=d_1$.
\end{thm}

The first part of the theorem is a consequence of a result of Landsberg and Manivel on the ideals of secant varieties, \cite{Lands-Mani-2004-IdealsSecVar}. The second part is obtained from a classification of the rank-semi-continuous varieties by A. V. Petukhov and the third author, \cite{Petukh-Tsan}. These results are somewhat phenomenological, but they suggest that secant varieties could fit well in the context of momentum and invariant theory. This is also evident from Wildberger's method for study of momentum images, where secant spaces play a key role. A further conceptual understanding of these relations remains as an open theme.

\subsubsection*{Acknowledgements}
This work was initiated at a seminar at Ruhr-Universit\"at Bochum in Winter 2014. We would like to thank Peter Heinzner and Alan Huckleberry for helpful discussions and support. T.M. is supported by Polish Ministry of Science and Higher Education ``Diamentowy Grant'' no. DI2013 016543 and ERC grant QOLAPS. E.H. and V.V.T. were supported by the project SFB/TR12 ``Symmetry and universality in mesoscopic systems'' at Ruhr-Universit\"at Bochum. In the latter stage of the work V.V.T was supported by DFG Schwerpunktprogramm 1388 ``Darstellungstheorie''.

\section{Setting and preliminaries}

Here we introduce the basic notation and conventions for the rest of the text. Let $K$ be a connected, simply connected, compact Lie group with semisimple Lie algebra ${\mk k}$. Let $G$ be the complexification of $K$ and ${\mk g}$ be its Lie algebra. We denote by $(.\vert.)$ the Killing form on ${\mk g}$, its restrictions to various subspaces, and induced forms on dual spaces. Let $T\subset K$ be a Cartan subgroup, $W=N_K(T)/T$ be the Weyl group and $H=T^\CC$ be the complexified Cartan subgroup. Let $\Lambda\subset{\mk h}^*$ be the integral weight lattice of $H$. Then the Killing form is positive definite on the real span $\Lambda_\RR$. We have a $W$-equivariant isomorphism $\Lambda_\RR\to i{\mk t}\cong{\mk t^*}$. Let $\Delta\subset\Lambda$ be the root system, so that
$$
{\mk g} = {\mk h}\oplus\left(\bigoplus\limits_{\alpha\in\Delta}{\mk g}^\alpha\right)
$$
Let $e_\alpha\in{\mk g}^\alpha$ be root vectors. We may, and do, assume that they are chosen so that
$$
( e_\alpha | e_\beta ) = \begin{cases} 1 & {\rm if}\;\; \alpha+\beta=0 \\ 0 & {\rm if}\;\; \alpha+\beta\ne 0 \end{cases} \;.
$$
Let $h_\alpha=[e_\alpha,e_{-\alpha}]$. Then, up to conjugating $K$ within $G$, we have
$$
{\mk k} = {\rm Span}\{ ih_\alpha, e_\alpha-e_{-\alpha},i(e_\alpha+e_{-\alpha}) \;;\; \alpha\in\Delta\} \;.
$$
Let $B\subset G$ be a Borel subgroup with $T\subset B$, let ${\mk t}^+$ be the corresponding Weyl chamber, $\Lambda^+$ the set of integral dominant weights, and $\Pi$ the set of simple roots. We consider the Killing form on ${\mk k}$ and use it to set an isomorphism between ${\mk k}$ and ${\mk k}^*$, and to embed ${\mk t}^*$ as a subspace of ${\mk k}^*$.

Let $V$ be a finite-dimensional $G$-module with a $K$-invariant Hermitian form, which we also denote by $\langle,\rangle$. Let $\Lambda(V)$ denote the set of weights and let $\lambda$ be the highest weight of $V$. We denote by $V^\nu$ the weight space of $V$ with weight $\nu\in\Lambda$ and, for $\nu\in\Lambda(V)$, we denote by $v^{\nu}$ an arbitrary weight vector with norm 1. The weight space decomposition $V=\oplus_{\nu\in\Lambda(V)}V^\nu$ is orthogonal. For any subset $M\subset \Lambda(V)$, we denote $V_M=\oplus_{\nu\in M}V^\nu$.

Let $\PP=\PP(V)$ be the projective space of $V$. We consider the following $K$-equivariant momentum map:
$$
\mu = \mu_K : \PP \lw {\mk k}^* \;,\quad \mu[v](\xi)=\frac{1}{i}\frac{\langle\xi v,v\rangle}{\langle v,v\rangle} \;,\; [v]\in\PP,\xi\in{\mk k}\;.
$$
We refer to Kirwan, \cite{Kirwan}, although the projective case was studied earlier, notably by Guillemin and Sternberg. Further references can be found in \cite{Kirwan}.

\begin{theorem}{\rm (Convexity theorem, \cite{Kirwan})}
The intersection
$$
\mu(\PP)^+=\mu(\PP)\cap\mk t^+ \;.
$$
of the momentum image with the Weyl chamber is a convex polytope.
\end{theorem}

This polytope is called the {\it momentum polytope of $\PP$}.\\

Now suppose that $V=V_\lambda$ is irreducible with highest weight $\lambda\in\Lambda^+$. Let $\XX=G[v^{\lambda}]\subset\PP$ denote projective $G$-orbit of the highest weight line. Then $\XX$ is a flag variety of $G$ and is characterized in this setting as: (1) the unique closed $G$-orbit in $\PP$; (2) the unique complex $K$-orbit in $\PP$; (3) the set in $\PP$ where $||\mu||^2$ takes its maximum value. We have $\XX^T=\{[v^{w\lambda}],w\in W\}$. The momentum map restricted to $\XX$ is bijective onto its image, $\mu(\XX)=K\lambda$, and is an isomorphism of K\"ahler manifolds, with respect to the Fubini-Study form induced from the projective embedding, on the one hand, and, on the other hand, the Kostant-Kirillov-Soriau K\"ahler form on coadjoint orbits of compact groups.

In this paper we study the connections between the projective geometry in $\PP$ related to the $K$-action and the convex Euclidean geometry in ${\mk k}^*$ related to the momentum map. Our approach is based on the method of Wildberger, \cite{Wildberger}, sketched in the next section. One key observation is that linear span corresponds to convex hull, i.e. for suitable subsets $A\subset\PP$ we have 
$$
\mu(\PP({\rm Span}(A)))={\mc C}(\mu(A))\;.
$$
In the present context this was formalized and applied by Wildberger, \cite{Wildberger}, whose method is presented below.

For any subset $A$ of ${\mk k}^*$ we denote: 

${\mc C}(A)={\rm Conv}(A)$ the convex hull of $A$;

${\rm Aff}(A)$ the affine span of $A$;

${\rm Cone}(A) = {\rm Span}_{\RR^+}(A)$ the convex cone spanned by $A$;

${\rm Cone}_a(A) = a+{\rm Cone}(A-a)$ for $a\in\mk k^*$, the affine convex cone with vertex $a$ generated by the rays from $a$ through $A$;

$A^+=A\cap{\mk t}^+$ the intersection of $A$ with the positive Weyl chamber;

$A^{++}=A\cap({\mk t}^+)^\circ$ the intersection of $A$ with the interior of the positive Weyl chamber;

If $A$ consists of 2 points $\xi,\xi'$, we denote by $E_{\xi'}^{\xi}$ or $[\xi,\xi']$ the convex hull, which is a line segment in this case.

Most of the calculations take place in the Cartan subalgebra and the weight lattice $\Lambda\subset{\mk t}^*$. Let $Q=\langle\Delta\rangle_\ZZ$ denote the root-lattice.

Let $A\subset\Lambda_\RR$ be an integral polytope and $a\in A$ be a vertex. If $F\subset A$ is a face and $a\in F$, we say that $F$ is a face of $A$ at $a$. If $E^a_{a'}$ is an edge at $a$ and $b\subset E^a_{a'}\setminus\{a\}$ is the closest $a$ integral point, we call $E^a_{b}$ an edge-generator of $A$ at $a$.

\subsection{The Weyl polytope}\label{Sect Weyl polytope}

In this section we collect some necessary facts and notation on the Weyl polytope ${\mc C}(W\lambda)$ for a given dominant weight $\lambda$. Since $W\lambda\subset \mu(\PP(V_\lambda))\cap{\mk t}^*\subset{\mc C}(W\lambda)$, the understanding of the Weyl polytopes is an essential step in the understanding of the momentum images of irreducible representations. Recall first of all that, for every $\lambda\in\Lambda$, $W\lambda$ is exactly the set of extreme points of ${\mc C}(W\lambda)$.

Fix $\lambda\in\Lambda^+$ and $\PP=\PP(V_\lambda)$ for this section. For $\alpha\in\Pi$ we denote by $\omega_\alpha$ the corresponding fundamental weight and write
$$
\lambda = \sum\limits_{\alpha\in\Pi} \lambda_\alpha\omega_\alpha \quad,\quad \lambda_\alpha = n_{\alpha,\lambda} = 2\frac{(\alpha|\lambda)}{(\alpha|\alpha)} \;.
$$
The weight $\lambda$ defines a partition of the set of simple roots
\begin{gather}\label{For Pi=Pilambda u Pi0}
\Pi=\Pi^{\lambda}\sqcup\Pi_0 \quad,\quad \Pi^{\lambda}=\{\alpha\in\Pi:\lambda_\alpha\ne0\} \quad,\quad \Pi_0=\{\alpha\in\Pi:\lambda_\alpha=0\} \;.
\end{gather}
This partition only depends on the face of the Weyl chamber to whose relative interior $\lambda$ belongs. The stabilizer $W_\lambda$ is generated by the simple reflections $s_\alpha,\alpha\in\Pi_0$. The edges of ${\mc C}(W\lambda)$ at $\lambda$ are
\begin{gather}\label{For EdgesCWlambdaAtlambda}
E_{ws_\alpha\lambda}^\lambda \quad,\quad \alpha\in\Pi^\lambda , w\in W_\lambda \;.
\end{gather}
(Recall that $E_p^q$ denotes the line segment between two points $p,q\in{\mk k}^*$.) The edge $E_{s_\alpha\lambda}^\lambda$ intersects the interior of the Weyl chamber ${\mk t}^{++}$ if and only if $(\alpha,\beta)\ne0$ for all $\beta\in\Pi_0$. In fact, $E_{s_\alpha\lambda}^\lambda$ belongs to the hyperplane $\mc H_\beta$ if and only if $(\alpha,\beta)=0$. We denote
$$
\eta_\alpha = \lambda-\frac12 \lambda_\alpha\alpha = E_{s_\alpha\lambda}^\lambda \cap\mc H_\alpha \quad,\quad {\rm for} \quad \alpha\in\Pi^\lambda \;.
$$
The faces of ${\mc C}(W\lambda)$ at $\lambda$ are determined by the edges they contain and hence are parametrized by the subsets of the set $\Pi$ of simple roots. The facets of ${\mc C}(W\lambda)$ at $\lambda$ are given by hyperplanes orthogonal to the fundamental rays corresponding to simple roots. If $\lambda$ is singular, then not every ray brings an actual facet of the Weyl polytope, some rays give rise to lower dimensional faces. For any $\nu\in\Lambda^+$, we denote
$$
\xi_{\alpha}(\nu) = (\RR^+\omega_\alpha) \cap {\mc C}(W_{\Pi\setminus\{\alpha\}}\nu) \quad,\quad \Xi(\nu)=\{\xi_\alpha(\nu);\alpha\in\Pi\}\;.
$$
We keep the simpler notation $\xi_\alpha$ for the values $\xi_{\alpha}(\lambda)$ at our fixed dominant weight.

\begin{rem}
We have
$$
{\mc C}(W\lambda)^+ = \left(\lambda-{\rm Cone}(W_\lambda\Pi^\lambda)\right)^+ \;.
$$
\end{rem}

\begin{prop}\label{Prop TriangCWlambda}
1) If $\#\Pi^\lambda>1$, the dominant part ${\mc C}(W\lambda)^+$ of the Weyl polytope admits the following simplicial subdivision:
$$
{\mc C}(W\lambda)^+ = S_0 \cup S_\lambda \cup \left( \bigcup\limits_{\alpha\in\Pi^\lambda} S_{\alpha,\lambda} \right) \;,
$$
where
\begin{gather}\label{For SimplicialSubCWlambda}
\begin{array}{rl}
S_0 & = {\mc C}\{0,\xi_\beta \;:\; \beta\in\Pi \} \\
S_\lambda & = {\mc C}\{\lambda,\xi_\beta \;:\; \beta\in\Pi \} \\
S_{\alpha,\lambda} & = {\mc C}\{\lambda,\eta_\alpha,\xi_\beta \;:\; \beta\in\Pi\setminus\alpha \} \;.\\
\end{array}
\end{gather}

2) If $\Pi^\lambda=\{\alpha\}$ is a singleton, then $S_\lambda$ is a face of $S_0$ and the simplicial subdivision takes the form ${\mc C}(W\lambda) = S_0\cup S_{\alpha,\lambda}$.
\end{prop}

\subsection{Wildberger's method}

Let $V$ be a finite dimensional unitary $K$-module. Recall that $\Lambda(V)$ denotes the set of weights of $V$. For $\nu\in\Lambda(V)$, we denote by $V^\nu\subset V$ the weight space and by $v^{\nu}$ an arbitrary weight vector with norm 1. The weight space decomposition $V=\oplus_{\nu\in\Lambda(V)}V^\nu$ is orthogonal. For any subset $M\subset \Lambda(V)$, we denote $V_M=\oplus_{\nu\in M}V^\nu$.

\begin{defin}{\rm (Wildberger, \cite{Wildberger})}

If $\alpha\in\Delta$ be a root, a weight $\nu\in\Lambda$ is called $\alpha$-distinct, if $\nu-s_\alpha\nu\ne\alpha$. A subset $M\subset\Lambda$ is called root-distinct if $M\cap(\Delta+M)=\emptyset$.
\end{defin}

In fact, the above definition differs from Wildberger's, formally but not conceptually. More precisely, Wildberger calls a dominant weight $\lambda$ root-distinct of its Weyl group orbit $W\lambda$ is a root-distinct set in our sense. We are however going to use more general sets than just Weyl group orbits, so we adopt the above formal definition.

The examples below and the following two lemmata play a fundamental role in our calculations.

\begin{ex}
The set of simple roots $\Pi$ is root-distinct. As a consequence, for $\lambda\in\Lambda^{++}$, the set $\lambda-\Pi$ is root-distinct; the weights in this set are the generators of the edges of the Weyl polytope ${\mc C}(W\lambda)$ at $\lambda$.
\end{ex}

\begin{lemma}{\rm (Wildberger \cite{Wildberger})}\label{Lemma WildbergerWlambdaRD}
Let $\lambda\in\Lambda^+$ and let $\lambda=\sum\limits_{\alpha\in\Pi} \lambda_\alpha\omega_\alpha$ be its expression in the basis of fundamental weights. The orbit $W\lambda$ is a root-distinct set if and only if $\lambda_\alpha\ne 1$ for all $\alpha\in\Pi$.
\end{lemma}

\begin{lemma}{\rm (Wildberger \cite{Wildberger})}\label{Lemma Wildberger}

If $M\subset\Lambda(V)$ is a root-distinct set, then $\mu(\PP(V_M))={\mc C}(M)$ and hence 

(a) $\mu(\PP(V_M))\subset{\mk t}^*$;

(b) $\mu(\PP(V^\nu))=\nu$, for $\nu\in\Lambda(V)$;

If the orbit $W\lambda$ is root distinct, then ${\mc C}(W\lambda)=\mu(\PP(V_{W\lambda}))=\mu(\PP)\cap{\mk t}^*$.
\end{lemma}

\begin{proof}
Suppose $M\subset \Lambda(V)$ is root-distinct. A direct calculation with the formula for $\mu$, using the fact that ${\mk g}^\alpha V^\nu\subset V^{\nu+\alpha}$, shows that
$$
\mu\left[\sum\limits_{\nu\in M} c_\nu v^\nu\right] = \sum\limits_{\nu\in M} |c_\nu|^2\nu \;,
$$
where we have taken a vector with norm 1, i.e. $\sum\limits_{\nu\in M}|c_\nu|^2=1$.

The last statement in the proposition follows immediately from the previous ones.
\end{proof}

As a consequence Wildberger proved the following.

\begin{theorem}{\rm (Wildberger \cite{Wildberger})}\label{Theo Wildberger}

Let $\lambda\in\Lambda^+$. The momentum image $\mu(\PP(V_\lambda))\subset{\mk k}^*$ is convex if and only if the Weyl group orbit $W\lambda$ is a root-distinct set of weights. When this is the case we have
$$
\mu(\PP(V_\lambda))={\mc C}(K\lambda) \quad,\quad \mu(\PP(V_\lambda))\cap{\mk t}^* = {\mc C}(W\lambda) = \mu(\PP(V_{W\lambda})) \;.
$$
\end{theorem}

\begin{coro}
Let $\XX=K[v^\lambda]\subset\PP=\PP(V)$ be the unique complex $K$-orbit. Then $\mu(\PP)$ is convex if and only if $\XX$ does not contain projective lines in $\PP$.
\end{coro}

\begin{lemma}\label{Lemma Wild Reducible}
{\rm (Wildberger, \cite{Wildberger})}
Let $V=V_1\oplus V_2$ be a direct sum of two unitary representation of a compact group $K$. Then the momentum image of $\PP(V)$ consists of all segments joining the momentum images of $\PP(V_1)$ and $\PP(V_2)$, i.e.
$$
\mu(\PP(V)) = \bigcup\limits_{(\xi_1,\xi_2)\in\mu(\PP(V_1))\times\mu(\PP(V_2))} [\xi_1,\xi_2] \;.
$$
\end{lemma}

\subsection{Unstable representations}\label{Sect 0 notin mu}

A representation is called unstable if the semistable locus is empty, i.e. $\PP=\PP_{us}$. This is equivalent to $\CC[V]^K=\CC$ and to $0\notin\mu(\PP)$. The irreducible unstable representations of semisimple compact groups are well known. They can be found as a particular case in Kac's classification of representations whose invariant ring is isomorphic t a polynomial ring. Since we shall use some properties of unstable representations, we list them below for convenience. We also give some results on reducible unstable representations. These representations form an exceptional class in our approach for calculation of momentum images and critical points of $||\mu||^2$.

The momentum image $\mu(\PP(V))$ for a given unitary representation $V$ of a semisimple compact group $K$ does not contain 0 if and only if the only $K$-invariant polynomials on $V$ are the constants, i.e.
$$
0\notin\mu(\PP(V)) \quad\tst\quad \CC[V]^K=\CC \;.
$$
In \cite{Kac-nilp-orb}, Kac has classified irreducible representations of semisimple groups with polynomial invariant rings. In particular, his tables contain the list of representations, where the only invariant polynomials are the constants. We extract the following.

\begin{prop}\label{Prop IrrepWithoutInvar}
The irreducible representations of semisimple groups, whose momentum image does not contain $0$ are the following, and their duals.
\begin{gather*}
\begin{array}{ll}
SU_m\otimes SU_n \quad&,\; n>m\geq1 \\
SU_n\otimes SO_m \quad&,\; n>m\geq3 \\
SU_n\otimes Sp_{2m} \quad&,\; n>2m\geq4, \quad{\rm or}\quad 1\leq n<2m\;{\rm with}\;n\;{\rm odd} \\
\Lambda^2 SU_{2n+1} \quad&,\; n\geq1 \\
Spin_{10} &\\
SU_2\otimes\Lambda^2 SU_5 &\\
SU_2\otimes\Lambda^2 SU_7 &\\
SU_2\otimes SU_3\otimes SU_5 &
\end{array}
\end{gather*}
In particular, if ${\mk k}$ is simple the list reduces to
\begin{gather*}
\begin{array}{ll}
SU_n \quad&,\; n\geq1 \\
Sp_m \quad&,\; m\geq2 \\
\Lambda^2 SU_{2n+1} \quad&,\; n\geq1 \\
Spin_{10} &
\end{array}
\end{gather*}
\end{prop}

\begin{rem}\label{Rem OnotinmuP}
We note for further use one common property of the representations from the above list. If $\lambda$ is the highest weight of an irreducible representation whose momentum image does not contain 0, then $\lambda$ is a sum of 1,2 or 3 fundamental weights, belonging to different simple components of the group $K$. We have $\Pi^\lambda=\Pi_1$ and the roots in $\Pi_1$ are short, whenever different root lengths occur. Thus
$$
0\in\mu(\PP) \quad\Lw\quad n_{\gamma,\alpha}\in\{0,-1\} \quad\textrm{for all $\alpha\in\Pi_1$ and $\gamma_\in\Pi_0$.}
$$
\end{rem}

As a consequence of Lemma \ref{Lemma Wild Reducible} we obtain the following.

\begin{lemma}
Let $V$ be a unitary representation of a semisimple compact group $K$. Let $V=\oplus m_\lambda V_\lambda$ be the decomposition of $V$ into isotypic components. If $0\notin\mu(\PP(V))$, then for each $\lambda$ with $m_\lambda\ne 0$ we have
$$
0\notin\mu(\PP(V_\lambda)) \;.
$$
Furthermore, if $V_\lambda$ is self-dual, we have $m_\lambda=1$. If $V_\lambda$ is not self-dual, we have
$$
m_\lambda< \#(W\lambda) \quad and\quad m_{\lambda^*} = 0 \;,
$$
where $\lambda^*$ denotes the highest weight of the dual representation $V_\lambda^*$ (we have $\lambda^*=-w_0\lambda$).
\end{lemma}

For representations without multiplicities we have the following.

\begin{prop}\label{Prop UnstableMultFree}
The only reducible multiplicity-free representations $V$ of a semisimple compact group $K$ for which $0 \notin \mu(\PP(V))$ are the $SU_{2n+1}$-representation on $(\CC^{2n+1})^* \oplus \Lambda^2\CC^{2n+1}$ and its dual.
\end{prop} 

The proof of this proposition is a somewhat lengthy calculation, which we have placed in the appendix.

\section{Momentum polytopes}

Let $\lambda\in\Lambda^+$ and $\PP=\PP(V_\lambda)$. We assume that the representation of ${\mk k}$ on $V_\lambda$ is faithful. As before, for $\alpha\in\Pi$ we denote by $\omega_\alpha$ the corresponding fundamental weight and write
$$
\lambda = \sum\limits_{\alpha\in\Pi} \lambda_\alpha\omega_\alpha \;.
$$
The goal of this section is to study the structure of the momentum polytope $\mu(\PP)^+=\mu(\PP)\cap{\mk t}^+$ and more generally $\mu(\PP)^{\mk t}$, concentrating at the vertex $\lambda$. We know that $W\lambda\subset\mu(\PP)^{\mk t}\subset{\mc C}(W\lambda)$. The two extreme cases do occur. We have $W\lambda=\mu(\PP)\cap{\mk t}^*$ if and only if $(K,V)$ is the natural representation of either $SU(n)$ or $Sp(n)$. The other extreme case is characterized by Wildberger's theorem. Recall that $\lambda$ is $\alpha$-distinct if and only if $\lambda_\alpha=1$ and $\lambda$ is orthogonal to $\alpha$ if and only if $\lambda_\alpha=0$. We consider a partition of the set of simple roots, refining (\ref{For Pi=Pilambda u Pi0}) according to the root-distinctness properties of $\lambda$:
\begin{gather}\label{For Part Pi lambda}
\begin{array}{lll}
\Pi &= \Pi_{\rm rd} \sqcup \Pi_1 \sqcup \Pi_0^\lambda \;,\;  & \Pi^\lambda=\Pi_{\rm rd}\sqcup\Pi_1 \\
\hline
\Pi_{\rm rd} &= \{\alpha\in\Pi: \lambda_\alpha\geq 2\} & \quad \textrm{root-distinct} \\
\Pi_{1} &= \{\alpha\in\Pi: \lambda_\alpha=1\} & \quad \textrm{non-root-distinct} \\
\Pi_{0} &= \{\alpha\in\Pi: \lambda_\alpha=0\} & \quad \textrm{orthogonal} \\
\end{array}
\end{gather}

\begin{rem}\label{Rem AboutPi0Pi1}
We have $\lambda\in\Lambda^{++}$ if and only if $\Pi_0=\emptyset$. Wildberger's theorem states, that $\mu(\PP)^{\mk t}={\mc C}(W\lambda)$ if and only if $W\lambda$ is root-distinct, i.e. $\Pi_1=\emptyset$.
\end{rem}

\subsection{The case of a regular highest weight}\label{The case of a regular highest weight}

We present here the case of a regular dominant weight. This case is contained in the general result which follow, but the combinatorics in this case are simpler, so we find it instructive to consider this case separately. We assume for technical reasons that $K$ has rank at least 3. The cases of rank 1 and 2 fit into the general construction, but require more involved statements.

Consider a regular dominant weight $\lambda\in\Lambda^{++}$. We have $\Pi^\lambda=\Pi$ and $W_\lambda=\{1\}$. Thus
$$
{\mc C}(W\lambda)^+ = \left( \lambda-{\rm Cone}(\Pi) \right)^+ = S_0 \cup S_\lambda \cup \left( \bigcup\limits_{\alpha\in\Pi} S_{\alpha,\lambda} \right) \;.
$$
Recall also that ${\mc H}_\alpha$ denotes the hyperplane in ${\mk t}^*$ orthogonal to $\alpha\in\Delta$.

\begin{theorem}\label{Theo MomentImLambdaReg}
Let $\lambda\in\Lambda^{++}$ and let $\Pi=\Pi_{\rm rd}\cup\Pi_1$ be the partition of the simple roots associated to $\lambda$, given in (\ref{For Part Pi lambda}). Denote
$$
\Theta^\lambda = \Pi_{\rm rd} \cup \left( \bigcup\limits_{\alpha\in\Pi_1} \alpha + s_{\alpha}(\Pi\setminus\{\alpha\}) \right) \quad,\quad \Upsilon_0={\rm Cone}(\Theta^\lambda) \quad,\quad \Upsilon_\lambda=\lambda-\Upsilon_0\;.
$$
Then the momentum polytope of $V_\lambda$ is obtained as the intersection of two cones as
$$
\mu(\PP)^+ = {\mk t}^+ \cap \Upsilon_\lambda \;.
$$
The vertices are obtained as follows. The minimal integral segments generating $\Upsilon_\lambda$ are
\begin{gather}\label{For EdgesAtLambdaReg}
E_{s_\alpha\lambda}^\lambda \;\;\textit{for}\;\; \alpha\in\Pi_{\rm rd}^\lambda \quad,\quad E_{s_\beta(\lambda-\gamma)}^{\lambda} \;\;\textit{for}\;\; \beta\in\Pi_1^\lambda \;,\; \gamma\in\Pi\setminus\{\beta\}\;.
\end{gather}
Denote $\eta_\alpha=E^\lambda_{s_\alpha\lambda}\cap\mc H_\alpha$, $\eta_{\beta}^\gamma=E_{s_\beta(\lambda-\gamma)}^{\lambda}\cap\mc H_\beta$. Then the momentum polytope of the projective representation $\PP(V_\lambda)$ is given by
$$
\mu(\PP(V_\lambda))^+={\mc C}\{ 0,\lambda,\eta_\alpha,\eta_\beta^\gamma,\xi_\delta \;;\; \alpha\in\Pi_{\rm rd}, \beta\in\Pi_1, \gamma\in\Pi\setminus\{\beta\}, \delta\in\Pi \} \;.
$$
\end{theorem}

\begin{proof}
First notice that the hypothesis implies $0\in\mu(\PP)$. Recall that we have assumed $\ell={\rm rank}(G)\geq 3$. Let $\alpha\in\Pi$ and let $F$ be the facet of ${\mc C}(W\lambda)$ at $\lambda$ orthogonal to $\xi_\alpha$. Also, let $L$ be the maximal Levi subgroup with simple roots $\Pi\setminus\{\alpha\}$. Let $V_F\subset V$ be the span of the weight spaces in $V$ whose weights belong to $F$; this is an $L$-submodule of $V$ whose highest weight is the restriction of $\lambda$ and whose weight polytope is $F$. Note that this restricted $\lambda$ is a regular dominant weight for $L$. The momentum images of regular representations of semisimple groups of rank at least 2 contain 0. Since ${\rm rank}(L)\geq 2$, we have $\xi_\alpha\in\mu_L(\PP(V_F))$, because $\xi_\alpha$ is in the center of ${\mk l}$ and plays the role of $0$ for the momentum map of the reductive group $L$. Thus
$$
\{\xi_\alpha ;\; \alpha\in\Pi\} \subset \mu(\PP) \;.
$$

Since $0,\lambda$ and $\xi_\alpha$, for all $\alpha$, belong to the momentum image $\mu(\PP)$, we have
$$
S_0 \cup S_\lambda \subset \mu(\PP) \;.
$$

For $\alpha\in\Pi$, we have $S_{\alpha,\lambda}\subset\mu(\PP)$ if and only if $\alpha\in\Pi_{\rm rd}$.

It remains to determine $S_{\beta,\lambda}\cap\mu(\PP)$ for $\beta\in\Pi_1$, i.e. $\lambda-s_\beta\lambda=\beta$. In this case, the edges from our list (\ref{For EdgesAtLambdaReg}) which intersect $S_{\beta,\lambda}$ are $E_{s_\beta(\lambda-\gamma)}^{\lambda}$ for $\gamma\in\Pi\setminus\{\beta\}$. The set of vertices of these edges is $s_\beta(\lambda-\Pi)$ with $\lambda=s_\alpha(\lambda-\alpha)$ being the common vertex and the only vertex in ${\mk t}^+$; this set is root-distinct, since $\Pi$ is root-distinct. By Wildberger's lemma \ref{Lemma Wildberger}, the convex hull of belongs to the momentum image, i.e.
$$
{\mc C}(s_{\alpha}(\lambda-\Pi))\in \mu(\PP) \;.
$$
We get
$$
S_{\beta,\lambda}\cap \Upsilon_\lambda={\mc C}\{\lambda,\eta_\beta^\gamma,\xi_\delta;\; \gamma,\delta\in\Pi\setminus\{\beta\} \} \;\subset\; \mu(\PP)^+ \;.
$$
Consider the complement
$$
X_{\beta,\lambda}=S_{\beta,\lambda}\setminus \Upsilon_\lambda \;.
$$
We have
$$
X_{\beta,\lambda} \subset {\mc C}(\{s_\beta\lambda\}\cup s_\beta(\lambda-\Pi)) \setminus {\mc C}(s_\beta(\lambda-\Pi)) \;.
$$
The polytope ${\mc C}(\{s_\beta\lambda\}\cup s_\beta(\lambda-\Pi))$ is an $\ell$-simplex and
$$
\Lambda(V) \cap {\mc C}(\{s_\beta\lambda\}\cup s_\beta(\lambda-\Pi)) = \{s_\beta\lambda\}\cup s_\alpha(\lambda-\Pi) \;.
$$
It follows that
$$
\Lambda(V) \cap X_{\beta,\lambda} \subset \Lambda(V)\cap{\mc C}(\{s_\beta\lambda\}\cup s_\beta(\lambda-\Pi)) \setminus {\mc C}(s_\beta(\lambda-\Pi))) = \{s_\beta\lambda\} \;.
$$
Thus $X_{\beta,\lambda}\cap\mu(\PP)\ne\emptyset$ would imply that $E_{s_\beta\lambda}^\lambda\subset\mu(\PP)$, which would come in contradiction with the fact that $\beta\in\Pi_1$. We conclude that
$$
S_{\beta,\lambda}\cap \mu(\PP)^+ = S_{\beta,\lambda}\cap \Upsilon_\lambda \;.
$$
This completes the proof.
\end{proof}

\subsection{A cone at $\lambda$ containing the momentum polytope}\label{A cone containing the image}

Recall that the edge-generators of $\mc C(W\lambda)$ at $\lambda$ are given by the set $\lambda-W_\lambda\Pi^\lambda$. Namely, they are $E_{\lambda-w\alpha}^\lambda$, for $\alpha\in\Pi^\lambda$, $w\in W_\lambda$.

\begin{lemma}\label{Lemma FencePredef}
The following hold:
\begin{enumerate}
\item[\rm (i)] The affine span ${\rm Aff}(W_\lambda\Pi^\lambda)$ is a hyperplane, given by
$$
{\rm Aff}(W_\lambda\Pi^\lambda) = {\rm Aff}(\Pi^\lambda)+{\rm Span}(\Pi_0) = \{\nu\in\mk t^*:(\nu|\chi_0)=||\chi_0||^2\} \;,
$$
where $\chi_0\in{\mk t}^*$ is defined by $(\chi_0|\alpha)=||\chi_0||^2$ and $(\chi|\gamma)=0$ for all $\alpha\in\Pi^\lambda$ and all $\gamma\in\Pi_0$.
\item[\rm (ii)] ${\rm Aff}(\lambda-W_\lambda\Pi^\lambda)\cap\mc C(W\lambda) = \mc C(\lambda-W_\lambda\Pi^\lambda)$.
\end{enumerate}
\end{lemma}

\begin{proof}
Fix one $\alpha\in\Pi^\lambda$ and consider the hyperplane
$$
\mc H = {\rm Aff}(\Pi^\lambda)+{\rm Span}(\Pi_0) = \alpha + {\rm Span}\{ \beta - \alpha, \gamma \;:\; \beta\in\Pi^\lambda\setminus\{\alpha\},\gamma\in\Pi_0 \} \;.
$$
Since $W_\lambda$ is generated by $s_\gamma$, $\gamma\in\Pi_0$ and $s_\gamma$ preserves ${\mc H}$, we have ${\rm Aff}(W_\lambda\Pi^\lambda)\subset{\mc H}$. The opposite inclusion follows from our assumption that the representation of ${\mk k}$ on $V_\lambda$ is faithful, so there are no simple components of $\Pi_0$ orthogonal to $\Pi^\lambda$. Then $0\in\mc H-\alpha$. Any normal vector $\chi\in(\mc H-\alpha)^\perp$ must satisfy $(\chi|\beta)=(\chi|\alpha)$ and $(\chi|\gamma)=0$ for all $\beta\in\Pi^\lambda$ and all $\gamma\in\Pi_0$. The normal $\chi_0$ belonging to ${\rm Aff}(W_\lambda\Pi^\lambda)$ is defined by further requiring $(\chi_0|\alpha)=(\chi_0|\chi_0)$. This proves part (i).

Part (ii) follows from the fact that ${\mc C}(W\lambda)$ is a convex polytope and $\lambda-W_\lambda\Pi^\lambda$ is exactly the set of vertices of the edge-generators of $\mc C(W\lambda)$ at $\lambda$.
\end{proof}

\begin{rem}
The hyperplane ${\rm Aff}(W_\lambda\Pi^\lambda)$ intersects all edges of $\mc C(W\lambda)$ at $\lambda$ at integral points. This follows directly from the definition, but is a remarkable and useful property of this hyperplane.
\end{rem}

We introduce some terminology, which we find helpful for the bookkeeping for the combinatorics to follow.

\begin{defin}
The set of roots $\Phi=W_\lambda\Pi^\lambda$ is called the {\rm root-fence of} $\lambda$. The convex hull $\tilde{\mc U}_\lambda=\mc C(\lambda-\Phi)$ is called the {\rm fence of $\lambda$}. The convex hull $\tilde{\mc U}_0=\mc C(\Phi)$ is called the {\rm fence of $\lambda$ at $0$}. The set ${\mc U}_\lambda=\mc C(\{\lambda\}\cup\tilde{\mc U}_\lambda)\setminus\tilde{\mc U}_\lambda$ is called the {\rm yard of $\lambda$}. The set ${\mc U}_0=\mc C(\{0\}\cup\tilde{\mc U}_0)\setminus\tilde{\mc U}_0$ is called the {\rm yard of $\lambda$ at $0$}. The intersection $\mu(\PP)^+\cap\ol{\mc U}_\lambda$ of the closed yard and the momentum polytope is called the {\rm house of $\lambda$}. The weights $s_\alpha\lambda$ for $\alpha\in\PP^\lambda$ are called the {\rm neighbours} of $\lambda$. The weights $s_\alpha\lambda$ for $\alpha\in\PP_1$ are called the {\rm root-neighbours} of $\lambda$. We use analogous terminology for $w\lambda$ with $w\in W$; for instance, $w\mc U_\lambda=\mc U_{w\lambda}$ is the yard of $w\lambda$.
\end{defin}

\begin{rem}
The convexity theorem for momentum polytopes implies that the house of $\lambda$ is a convex polytope. The above lemma implies that the yard ${\mc U}_\lambda$ is a connected component of $\mc C(W\lambda)\setminus\tilde{\mc U}_\lambda$. By definition, a neighbour $s_\alpha\lambda$ is a root-neighbour if and only if $\lambda-s_\alpha\lambda=\alpha\in\Delta$. By Wildberger's theorem $E^\lambda_{s_\alpha\lambda}\in\mu(\PP)$ if and only if $s_\alpha\lambda$ is not a root-neighbour.
\end{rem}

The following lemma asserts that $\lambda$ is alone in its yard. Recall that $Q\subset\Lambda$ denotes the root lattice.

\begin{lemma}\label{Lemma wlambdainUwlambda}
We have $Q\cap\mc U_0=\{0\}$ and $\mc U_{w\lambda}\cap\Lambda(V)=\{w\lambda\}$ for $w\in W$.
\end{lemma}

\begin{proof}
Notice that $Q\cap \mc U_0\subset Q_+$ since $W_\lambda\Pi^\lambda\subset\Delta^+$ and $\mc U_0\subset{\rm Cone}(W_\lambda\Pi^\lambda)$. By the definition of $\chi_0$ we have $(\chi_0|\nu)\geq 0$ for all $\nu\in Q_+$. Thus $\nu\in{\mc U}_0$ implies $(\chi_0|\nu)=0$ and hence $\nu\in Q_+(\Pi_0)$. But $Q_+(\Pi_0)\cap{\rm Cone}(W_\lambda\Pi^\lambda)=\{0\}$, where $Q_+(\Pi_0)$ is the monoid generated by $\Pi_0$. Hence $\nu=0$. 

Since $\Lambda(V)\subset \lambda-Q$ and $\Lambda(V)$ is $W$-invariant, we get $\mc U_{w\lambda}\cap\Lambda(V)=\{w\lambda\}$ for $w\in W$
\end{proof}

The following lemma asserts that the house of $\lambda$ does not intersect any neighbour's yard. The houses could meet at the fence.

\begin{lemma}
If $\alpha\in\Pi^\lambda$, then
$$
\mu(\PP)^+ \cap \mc U_{s_\alpha\lambda} = \emptyset \;.
$$ 
\end{lemma}

\begin{proof}
Let $\alpha\in\Pi^\lambda=\Pi_{\rm rd}\sqcup\Pi_0$. We have $\alpha\in\Pi_{\rm rd}$ if and only if $\mc U_{s_\alpha\lambda}\cap\mk t^+=\emptyset$. Suppose $\alpha\in\Pi_1$. The set $\mc U_{s_\alpha\lambda}$ is the intersection of ${\mc C}(W\lambda)$ with an open half-space. Namely, the open half-space containing $s_\alpha\lambda$ and bounded by the hyperplane $s_\alpha(\lambda-{\rm Aff}(W_\lambda\Pi^\lambda))$. It follows that, if the set $\mu(\PP)^+ \cap \mc U_{s_\alpha\lambda}$ is nonempty, it must contain an edge $E$ of $\mu(\PP)^+$ of the form $E=E^\lambda_\nu$ with $\nu\in \mc U_{s_\alpha\lambda}\cap \Lambda(V)$. From Lemma \ref{Lemma wlambdainUwlambda}, we know that $\mc U_{s_\alpha\lambda}\cap\Lambda(V)=s_\alpha\lambda= \lambda-\alpha$. But $E^\lambda_{s_\alpha\lambda}$ does not belong to $\mu(\PP)$ because $\alpha\in\Pi_1$. 
\end{proof}

\begin{coro}\label{Coro NoCommonYards}
We have
$$
\mu(\PP)^+ \subset {\mc C}(W\lambda)^+ \setminus \left( \bigcup\limits_{\alpha\in\Pi_1} \mc U_{s_\alpha\lambda}^+ \right) = \left\{ \nu\in\mc C(W\lambda)^+: (\nu|s_\alpha\chi_0)\leq(\lambda|s_\alpha\chi_0) \;,\;{\rm for}\;\alpha\in\Pi_1 \right\} \;.
$$
The momentum image does not contain any intersections of yards, i.e.
$$
\mu(\PP)^{\mk t} \subset \mc C(W\lambda) \setminus \left( \bigcup\limits_{w,w'\in W} \mc U_{w\lambda}\cap\mc U_{w'\lambda} \right)
$$
\end{coro}

We shall see that in a number of cases these conditions determine exactly the momentum image. The above corollary implies that the momentum polytope $\mu(\PP)^+$ is contained in the intersection of the Weyl chamber with a rational polyhedral cone. We now proceed to describe this cone. We will make a $W_\lambda$-invariant construction, considering the momentum polytopes of $\PP(V_\lambda)$ relative to all Weyl chambers containing $\lambda$. This will allow us to have integral generators for the requested cone at $\lambda$.

We start with the cone at $\lambda$ generated by the Weyl polytope, ${\rm Cone}_\lambda(W\lambda)=\lambda-{\rm Cone}(W_\lambda\Pi^\lambda)$. We consider the subcone generated by the polytope where all yards of root-neighbours of $\lambda$ are removed:
\begin{gather}\label{For Upsilonlambda}
\Upsilon_\lambda = {\rm Cone}_\lambda \left( \mc C(W\lambda) \setminus ( \bigcup\limits_{\alpha\in\Pi_1,w\in W_\lambda} \mc U_{ws_\alpha\lambda} ) \right)  \;.
\end{gather}

In the next lemma we give two descriptions of the cone $\Upsilon_\lambda$. One by the inequalities bounding $\Upsilon_\lambda$ within ${\rm Cone}_\lambda(W\lambda)$. The other, under a condition, as a cone at $\lambda$ generated by a specific part of the fence.

\begin{lemma}
The following hold:
\begin{enumerate}
\item[\rm (i)] $\Upsilon_\lambda = \left\{ \nu\in {\rm Cone}_\lambda(W\lambda) : (\nu|ws_\alpha\chi_0)\leq(\lambda|ws_\alpha\chi_0) , \forall\alpha\in\Pi_1,\forall w\in W_\lambda \right\}$.
\item[\rm (ii)] If $\tilde{\mc U}_\lambda^+\ne\emptyset$, then $\Upsilon_\lambda = {\rm Cone}_\lambda(W_\lambda \wt{\mc U}_\lambda^+)$.
\end{enumerate}
\end{lemma}

\begin{proof}
For given $\alpha\in\Pi_1$ and $w\in W_\lambda$, the hyperplane 
$$
{\rm Aff}(ws_\alpha(\lambda-W_\lambda\Pi^\lambda)) = \{ \nu\in\mk t^*: (\nu|ws_\alpha\chi_0)=(\lambda|ws_\alpha\chi_0) \}
$$
divides $\mc C(W\lambda)$ in two connected components: $\mc U_{ws_\alpha\lambda}$ and its complement. Thus
$$
\mc C(W\lambda)\setminus\mc U_{ws_\alpha\lambda} = \{\nu\in \mc C(W\lambda) : (\nu|ws_\alpha\chi_0)\leq(\lambda|ws_\alpha\chi_0) \}
$$
This implies part (i). For part (ii), notice that ${\rm Cone}_\lambda(W\lambda)={\rm Cone}_\lambda(\wt{\mc U}_\lambda)$ and
$$
\ol{\mc U}_\lambda^+=\mc C(\{\lambda\}\cup\tilde{\mc U}_\lambda^+) \cup \left( \bigcup\limits_{\alpha\in\Pi_1,w\in W_\lambda} \ol{\mc U}_{ws_\alpha\lambda}^+  \right) \;.
$$
Now the statement follows from the fact that $\mc U_{ws_\alpha\lambda} \cap \Upsilon_\lambda=\emptyset$ for $\alpha\in\Pi_1,w\in W_\lambda$.
\end{proof}

\begin{theorem}\label{Theo muPPinUpsilonlambda}
For any $\lambda\in\Lambda^+$ we have
$$
W_\lambda(\mu(\PP)^+) \subset \Upsilon_\lambda \;.
$$
In particular,
$$
\mu(\PP)^+\subset\Upsilon_\lambda^+\;.
$$
\end{theorem}

\begin{proof}
The theorem follows directly from Corollary \ref{Coro NoCommonYards} and the $W_\lambda$-invariance of $\Upsilon_\lambda$ and $\mu(\PP)$.
\end{proof}

\begin{coro} We have
$$
\mu(\PP)\cap \ol{\mc U}_\lambda \subset \Upsilon_\lambda \cap \ol{\mc U}_\lambda = \mc C(\{\lambda\}\cup W_\lambda\tilde{\mc U}_\lambda^+) \;.
$$
\end{coro}

The question arises: when do we have $\mu(\PP)^+=\Upsilon_\lambda^+$? We have seen that the equality holds for strictly dominant weights. It is not hard to see that equality also holds for some highly singular representations. For instance, for the natural representation of $SL_n$, where $\lambda=\omega_1$, we get $\Upsilon_{\omega_1}=\{\omega_1\}=\mu(\PP)^+$. However, an example below shows that, for the natural representation of $Sp_4$ the inclusion is proper. We address the question more systematically in the next section. First we make some remarks on $\Upsilon_\lambda$.

\begin{lemma}\label{Lemma 0inUpsilon}
We have $0\in\Upsilon_\lambda$ if and only if $(\lambda|\chi_0)\geq||\chi_0||^2$.
\end{lemma}

\begin{proof}
Consider the shifted cone $\Upsilon_0=\lambda-\Upsilon_\lambda$. Notice that $0\in\Upsilon_\lambda$ if and only if $\lambda\in\Upsilon_0$. Let ${\mc H}_{\Pi_0}^+ = \cap_{\gamma\in\Pi_0} \mc H_\gamma^+$ be the face of the Weyl chamber $\mk t^+$ to which $\lambda$ belongs. Then $\mc H_{\Pi_0}^+\subset{\rm Cone}(W_\lambda\Pi^\lambda)$. On the other hand, by the definition of $\Upsilon_0$, we have 
$$
\Upsilon_0=\left\{ \nu\in {\rm Cone}(W_\lambda\Pi^\lambda) \;:\; (ws_\alpha\chi_0|\nu)\geq0 \;\;\textrm{for all}\;\;\alpha\in\Pi_1, w\in W_\lambda\right\}\;.
$$
Hence $\lambda\in\Upsilon_0$ if and only if $(ws_\alpha\chi_0|\lambda)\geq0$ for all $\alpha\in\Pi_1$ and $w\in W_\lambda$. In fact all these inequalities are the same. For $\alpha\in\Pi_1$, we have $(\chi_0|\alpha)=||\chi_0||^2$ by the definition of $\chi_0$ and hence
$$
(ws_\alpha\chi_0|\lambda) = (\chi_0|s_\alpha\lambda)=(\chi_0|\lambda-\alpha) = (\chi_0|\lambda)-||\chi_0||^2 \;.
$$
This implies the statement of the lemma.
\end{proof}

\begin{rem}
Notice that the above lemma can be reformulated as
$$
0\notin\Upsilon_\lambda \;\;\textrm{if and only if}\;\; \lambda\in \mc U_0\;.
$$
In particular, $0\notin\Upsilon_\lambda$ implies $\lambda\notin Q$.
\end{rem}

The following example shows that the momentum polytope $\mu(\PP)^+$ may or may not be equal to $\Upsilon_\lambda^+$.

\begin{ex}\label{Ex Upsilon for Sp4}
Let $K=Sp_4$. We consider the two fundamental representations of $K$.

1) Let $\lambda=\omega_1$, so that $V=\CC^4$ is the natural representation. We have $\mu(\PP)^+=\{\lambda\}$, because the symplectic group acts transitively in its natural representation, i.e. $\PP=K[v^\lambda]$ is a single $G$-orbit. We have $\Pi_{\rm rd}=\emptyset$, $\Pi_1=\{\alpha_1\}$, $\Pi_0=\{\alpha_2\}$. We obtain $\Upsilon_\lambda=\lambda-\RR_+2\lambda$. Thus $\Upsilon_\lambda^+=E^\lambda_0$ consists of a single edge and $\mu(\PP)^+\ne\Upsilon_\lambda^+$. Notice also that $0\in\Upsilon_\lambda^+$ but $0\notin\mu(\PP)$. We actually have $\lambda=\chi_0$.

2) Let $\lambda=\omega_2$, so that $V=\CC^5$ reduces to the natural representation of $SO_5$. We have $\Pi_{\rm rd}=\emptyset$, $\Pi_1=\{\alpha_2\}$, $\Pi_0=\{\alpha_1\}$. We obtain $\Upsilon_\lambda=\lambda-\RR_+2\lambda$. Thus $\Upsilon_\lambda^+=E^\lambda_0$ consists of a single edge. We know for the theorem that $\mu(\PP)^+\subset\Upsilon_\lambda^+$. Equality holds, because $-\lambda\in\Lambda(V)$ and $\lambda-(-\lambda)=2\lambda=2\alpha_1+2\alpha_2\notin\Delta$, so $E_{-\lambda}^\lambda$ is root-distinct and belongs $\mu(\PP)$. In this case we also have $\lambda=\chi_0$.
\end{ex}

\subsection{The intersection of two cones}\label{The intersection of two cones}

The goal of this section is to prove the following.

\begin{theorem}
Suppose that ${\rm rank}(\mk k)>1$ and $(\mk k,V_\lambda)$ is a faithful irreducible representation. If $\Xi(\lambda)\subset\mu(\PP)$ and $n_{\gamma,\alpha}\ne-1$ for all $\alpha\in\Pi_1$ and $\gamma\in\Pi_0$, then
$$
\mu(\PP)^+=\Upsilon_\lambda^+ \;.
$$
\end{theorem}

\begin{proof}
We shall use the simplicial subdivision (\ref{For SimplicialSubCWlambda}) of $C(W\lambda)$, and determine the intersection of $\mu(\PP)^+$ with each simplex. Since $\lambda,0\in\mu(\PP)$ and $\Xi=\Xi(\lambda)\subset\mu(\PP)$ we have
$$
S_0\cup S_\lambda\subset\mu(\PP)^+ \;.
$$
It remains to consider the simplices $S_{\alpha,\lambda}$ for $\alpha\in\Pi^\lambda=\Pi_{\rm rd}\cup\Pi_{1}$. For $\alpha\in\Pi_{\rm rd}$ we have $\eta_\alpha\in\mu(\PP)$ and hence $S_{\alpha,\lambda}\subset\mu(\PP)^+$. Let $\alpha\in\Pi_1$. We know that $\mu(\PP)^+\cap\mc U_{s_\alpha\lambda}=\emptyset$. Hence
$$
\mu(\PP)^+\cap S_{\alpha,\lambda}\subset S_{\alpha,\lambda}\setminus\mc U_{s_\alpha\lambda}^+ = \mc C(\tilde{\mc U}_{s_\alpha\lambda}^+\cup\Xi\setminus\{\xi_\alpha\}\}) \;.
$$
The assumption that $\Xi\in\mu(\PP)$ implies that $\mc C(\{\lambda\}\cup\Xi\setminus\{\xi_\alpha\})\subset\Upsilon_\lambda$, and hence the intersection of $\tilde{\mc U}_{s_\alpha\lambda}$ with the boundary of the Weyl chamber is contained in a single facet of $\mk t^+$:
$$
\tilde{\mc U}_{s_\alpha\lambda}\cap\partial(\mk t^+)\subset\mc C(\{\eta_{\alpha}\}\cup\Xi\setminus\{\xi_\alpha\})\subset\mc H_\alpha^+ \;.
$$
To show that $\mu(\PP)^+\cap S_{\alpha,\lambda}= \Upsilon_\lambda\cap S_{\alpha,\lambda}=\mc C(\tilde{\mc U}_{s_\alpha\lambda}^+\cup\Xi\setminus\{\xi_\alpha\}\})$ it suffices to show that all edges connecting $\lambda$ to vertices of $\tilde{\mc U}_{s_\alpha\lambda}$ belong to the momentum image. Recall that the vertices of $\tilde{\mc U}_{s_\alpha\lambda}$ form the set $\{s_\alpha(\lambda-W_\lambda\Pi^\lambda)\}$. The root-distinctness properties of this set are the same as these of $W_\lambda\Pi^\lambda$. The requested edges are of the form $E^\lambda_{s_\alpha(\lambda-w\beta)}=s_\alpha(\lambda-E^{\alpha}_{w\beta})$ with $\beta\in\Pi^\lambda$, $w\in W_\lambda$. All these edges are root-distinct, as we show in the following lemma.

\begin{lemma}\label{Lemma r-d-inWlambdaPilambda}
The following hold for the set $W_\lambda\Pi^\lambda$:

{\rm (i)} For every two distinct roots $\alpha,\beta\in\Pi^\lambda$ and every $w,w'\in W_\lambda$, we have $w\alpha-w'\beta\notin\Delta$, i.e. the weights $w\alpha$ and $w'\beta$ are root-distinct.

{\rm (ii)} Let $\alpha\in\Pi_1$ be such that $n_{\gamma,\alpha}=-1$ for all $\gamma\in\Pi_0$. Then the set $W_\lambda\alpha$ is root-distinct.
\end{lemma}

\begin{proof}
Since $\Delta$ is preserved by $W$, it is sufficient to show that $\alpha-w\beta\notin\Delta$ for $w\in W_\lambda$. This holds, because $w\beta$ is a positive root in the span of $\Pi_0\cup\beta$, so $\alpha-w\beta$ has both positive and negative coordinates in the basis of simple roots. This proves (i).

To prove (ii), we consider $\alpha$ as a weight of $K_\lambda$, where $K_\lambda$ is the stabilizer of $\lambda$ in $K$. Since $\Pi_0$ is a set of simple roots for $K_\lambda$ and $(\alpha|\gamma)\leq 0$ for all $\gamma\in\Pi_0$, $\alpha$ is an antidominant weight for $K_\lambda$ with the basis $\Pi_0$. Now the assumption $n_{\gamma,\alpha}\ne -1$ implies that $W_\lambda\alpha$ is root-distinct, by Wildberger's lemma given here as Lemma \ref{Lemma WildbergerWlambdaRD}.
\end{proof}

This lemma implies that all edges of the form $E^\lambda_{s_\alpha(\lambda-w\beta)}$ with $\beta\in\Pi^\lambda$, $w\in W_\lambda$ are root-distinct. According to the above remarks, we may conclude that
$$
\mu(\PP)^+\cap S_{\alpha,\lambda} = \Upsilon_\lambda\cap S_{\alpha,\lambda} = \mc C(\tilde{\mc U}_{s_\alpha\lambda}^+\cup\Xi\setminus\{\xi_\alpha\}\})
$$
and hence
$$
\mu(\PP)^+=\Upsilon_\lambda^+ \;,
$$
which completes the proof of the theorem.
\end{proof}

\section{Critical points of $||\mu||^2$}\label{Critical points}

According to Kirwan, \cite{Kirwan}, the function $||\mu||^2:\PP\lw\RR$ has certain key properties of a Morse-Bott function on $\PP$ and can be used to define an equivariant stratification of the space. To this end one needs to determine the critical points of $||\mu||^2$.

\begin{defin}
Let $\xi\in{\mk t}^*$ be such that $\mu^{-1}(\xi)$ contains critical points of $||\mu||^2$. Then we call $\xi$ an intermediate critical value of $||\mu||^2$. We denote by $C_\xi$ the set of critical points in $\mu^{-1}(\xi)$. 
\end{defin}

\begin{defin}
For $M\subset\Lambda$, let $\xi_M$ denote the closest to the origin point in ${\mc C}(M)$.
\end{defin}

\begin{prop}{\rm (Kirwan, \cite{Kirwan})}\label{prop_kirwan}

If $\xi\in{\mk t}^*$ is an intermediate critical value of $||\mu||^2$, then $\xi=\xi_M$ for some $M\subset\Lambda(V)$.
\end{prop}

In the following, we denote
$$
{\Xi}=\{\xi_M;\; M\subset \Lambda(V)\}
$$
The calculation of the elements of $\Xi$ is theoretically a rather straightforward task, as it can be done, for example, by applying the method of Lagrange multipliers to find the minimum of the vector's squared norm, as follows. We can write an element of $\mathrm{Conv}(M)$ as
\begin{equation*}
\xi=\sum_{\lambda\in M}a_\lambda \lambda,\ \ a_\lambda \geq 0,\ \ \sum_{\lambda\in M}a_\lambda=1.
\end{equation*}
Then the condition for the minimum of the squared norm reads
\begin{equation*}
||\xi_M||^2=\min_{\{a_\lambda\},\ \lambda\in M}\sum_{\lambda, \lambda'\in M}a_\lambda a_{\lambda'}(\lambda|\lambda').
\end{equation*}
The method of Lagrange multiplies requires consideration of the Lagrangian
\begin{equation}
\mathcal{L}(\{a_\lambda\};\delta):=\delta\left(\sum_{\lambda\in M}a_\lambda-1\right)+\sum_{\lambda, \lambda'\in M}a_\lambda a_{\lambda'}(\lambda|\lambda').
\label{lagrangian}
\end{equation}
The conditions for the set of parameters $\{a_\lambda\}_{\lambda\in M}$ to minimize the norm of $\xi\in\mathrm{Conv}(M)$ read
\begin{equation*}
\frac{\partial\mathcal{L}}{\partial a_\lambda}=0 \;\;\textrm{for all}\;\;{\lambda\in M}
\end{equation*}
which, applied to equation (\ref{lagrangian}), gives a set of linear equations
\begin{equation*}
\delta+2\sum_{\lambda'\in M}a_{\lambda'}(\lambda|\lambda')=0\;\; \textrm{for all}\;\;{\lambda\in M}
\end{equation*}
or, equivalently,
\begin{equation*}
(\xi_M|\lambda)=-\frac{\delta}{2}=||\xi_M||^2 \;\;\textrm{for all}\;\;{\lambda\in M}\;.
\end{equation*}
There is also one simplification for the number of elements of the sets $M$ to be made. Note that a convex hull of weights from $M$ can be regarded as a convex subset of $\RR^{\mathrm{dim}(\mk{t})}$. Therefore, it can be divided into simplices of dimension $\mathrm{dim}(\mk{t})$ or smaller that have disjoint interiors. Moreover, the closest to zero point of a fully dimensional convex polytope belongs to its boundary, so it is enough to consider simplices of maximal dimension $\mathrm{dim}(\mk{t})-1$. Such simplices are spanned on $\mathrm{dim}(\mk{t})$ linearly independent vertices, hence the following

\begin{coro}  \label{coro_minimal_weights}
Each $\xi_M$ can be found as the closest to zero point of maximally $\mathrm{dim}(\mk{t})$ linearly independent weights.
\end{coro}

We will next give a characterization of the critical sets of $||\mu||^2$ in terms of sets $Z_\xi$, which we define, following Kirwan \cite{Kirwan}. The sets $Z_\xi$ are defined in terms of the momentum map stemming from the action of the maximal torus, $T\subset K$. The group $T$ is not semisimple, therefore $\mu_T$ is not unique. However, it can be made unique by requiring the condition $\mu_T=pr_{\mk{t}}\circ\mu$, where $pr_{\mk{t}}$ is the orthogonal projection from ${\mk k}$ onto ${\mk t}$. For $\xi\in\mk{t}$ define $T_\xi=\overline{\{e^{\xi t}: t\in\RR\}}$.

\begin{defin}{\rm(Kirwan, \cite{Kirwan})}
For arbitrary $\xi\in\mk{t}$, let $Z_\xi$ denote the union of those connected components of the $T_\xi$-fixed points in $\PP$, for which $(\mu_T(x)|\xi)=||\xi||^2$.
\end{defin}

The sets $Z_\xi$ are $T$-invariant symplectic manifolds \cite{Kirwan}. Since $\mu$ is $K$-equivariant and each coadjoint orbit of $K$ intersects ${\mk t}^+$ at exactly one point, for every $[v]\in\PP$ we have $\mu(K[v])\cap \mk{t}^+=\{\xi\}$ if and only if $\mu_T([u])=\xi$ for some $u\in Kv$. Recall that, if $\mu'$ is any component of $\mu$, then $[v]$ is a critical point of $||\mu'||^2$ if and only if the fundamental vector field, $\wh{\mu'([v])}$, vanishes at the point $[v]$, it follows that
\begin{prop}{\rm(Kirwan, \cite{Kirwan})}
All critical sets of $||\mu||^2$ are of the form $C_\xi=K(Z_\xi\cap\mu^{-1}(\xi))$, where $\xi$ runs over the set $\Xi\cap\mk{t}^+$.
\end{prop}

The following lemma describes the sets $Z_\xi$ in the general case of a projective space acted by a compact group $K$ via a unitary representation.

\begin{lemma}
\label{lemma:z_beta}
Let $\Lambda_\xi=\{\lambda\in\Lambda(V):\ (\lambda|\xi)=||\xi||^2\}$ for arbitrary $\xi\in\mk{t}$ and let $V_\xi=\bigoplus_{\nu\in\Lambda_\xi}V^\nu$. Then $Z_\xi=\PP(V_\xi)$.
\end{lemma}

\begin{proof}
First, we show that $\PP(V_\xi)\subset Z_\xi$. As a sum of weight spaces, $V_\xi$ is clearly a $T$-submodule of $V$. Atiyah's convexity theorem, cf. \cite{Atiyah}, implies that $\mu_T\left(\PP(V_\xi)\right)=\mathrm{Conv}\left[\mu_T\left(\PP(V_\xi)^T\right)\right]$, where $\PP(V_\xi)^T$ denotes the set of $T$-fixed points in $\PP(V_\xi)$.  Since $\mu_T\left(\PP(V_\xi)^T\right)=\Lambda_\xi$, for any $v\in V_\xi$ we have
\begin{equation*}
(\mu_T([v])|\xi)=\sum_{\nu\in\Lambda_\xi} a_\nu (\nu|\xi)=||\xi||^2\left(\sum_{\nu\in\Lambda_\xi} a_\nu\right)=||\xi||^2.
\end{equation*}
The $T$-invariance also implies that each point of $\PP(V_\xi)$ is a fixed point for the action of $T_\xi\subset T$.
The last thing to show is that $\PP(V\setminus V_\xi)\cap Z_\xi=\emptyset$. Assume that this statement is not true, i.e. there exists $[v]\in Z_\xi$ such that
\begin{equation}
v=z_1u+z_2w,\ u\in V_\xi,\ w\notin V_\xi,
\label{v_form}
\end{equation}
with $\langle u,u\rangle=\langle w,w\rangle=1 ,\ z_1,z_2\in\CC,\ |z_1|^2+|z_2|^2=1$.
The point $[v]$ being a $T_\xi$-fixed point is equivalent the vector $v$ being an eigenvector of $\xi$, i.e. $\xi v=r v$, $r\in\CC$. Note that all vectors from $Z_\xi$ belong to the same eigenspace of $\xi$. To see this, let us calculate the image of $[v]$ under the $\xi$-component of $\mu_T$
\begin{equation}
\mu_{T}([v])(\xi)=\frac{1}{i}\langle\xi v,v\rangle=\frac{1}{i}\langle r v,v\rangle=-ir^*.
\label{v_image}
\end{equation}
On the other hand, $\mu_{T}([v])(\xi)=(\mu_{T}([v])|\xi)=||\xi||^2$, which with equation (\ref{v_image}) implies that $r=-i||\xi||^2$. Applying this result to $v$ from equation (\ref{v_form}) and using the fact that $v$ and $u$ belong to the same eigenspace of $\xi$, one gets that
\begin{equation}
\xi.w=-i||\xi||^2w.
\label{w_eigen}
\end{equation}
Decompose the vector $w$ as a sum of normed weight vectors:
\begin{equation}
w=\sum_{\nu\notin\Lambda_\xi} c_\nu w^\nu,\ \sum_{\nu\notin\Lambda_\xi}|c_\nu|^2=1.
\label{w_decomp}
\end{equation}
Putting the vector $w$ from eq. (\ref{w_decomp}) into eq. (\ref{w_eigen}), one obtains
\begin{equation}
\sum_{\nu\notin\Lambda_\xi}(-i)c_\nu\left[(\xi|\nu)-||\xi||^2\right]w^\nu=0,
\end{equation}
which means that for the weights $\nu\notin\Lambda_\xi$, for which $c_\nu\neq0$, one has $(\xi|\nu)=||\xi||^2$. This is in contradiction with the definition of $\Lambda_\xi$, therefore the only possibility is that $w=0$.
\end{proof}

\begin{coro}
\label{coro_centraliser}
Let $\xi\in{\Xi}$, let $K_\xi$ be the centralizer of $\xi$ in $K$ and $K'_\xi$ be the semisimple part of $K_\xi$. The group $K'_\xi$ acts unitarily on $V_\xi$. We have $Z_\xi=\PP(V_\xi)$ and a momentum map $\mu_{K'_\xi}:Z_\xi\lw({\mk k}'_\xi)^*$. Then $\xi$ is an intermediate critical value of $||\mu||^2$ if and only if $0\in\mu_{K'_\xi}(Z_\xi)$.
\end{coro}


The sets $Z_\xi$ can be very big, e.g. when $\xi=0$, $Z_\xi$ is equal to the whole projective space. Another extreme case is $\xi=\lambda$, where $Z_\xi=[v^\lambda]$ is a single point. The structure of $\Lambda_\xi$ as a subset of the weight lattice depends on whether $\xi$ is regular, i.e. belongs to the interior of a Weyl chamber, or not. In particular we have to following lemma.

\begin{lemma}
\label{lemma:regular}
If $\xi\in\mk{t}$ is regular, then $\Lambda_\xi:=\{\lambda\in\Lambda(V):\ (\lambda|\xi)=||\xi||^2\}$ is root-distinct.
\end{lemma}
\begin{proof}
Assume that $\Lambda_\xi$ is not root distinct, i.e. there exists $\alpha\in\Delta$ such that $\lambda-\lambda'=\alpha$ for some $\lambda,\ \lambda'\in\Lambda_\xi$. The fact that $(\lambda|\xi)=(\lambda'|\xi)$ implies that $(\alpha|\xi)=0$. This means that $\xi$ is not regular.
\end{proof}

\begin{prop}
If $\xi\in\Xi$ is regular, then the critical set $C_\xi$ is nonempty, i.e. $\xi$ is an intermediate critical value of $||\mu||^2$.
\end{prop}

\begin{proof}
Having Lemma \ref{lemma:regular} in hand, we know that $\mu(Z_\xi)={\mc C}(\Lambda_\xi)$. An exemplary state from  $Z_\xi\cap\mu^{-1}(\xi)$ can be constructed using Lemma \ref{Lemma Wildberger}. Namely, if $\xi=\sum_{\lambda\in\Lambda_\xi}a_\lambda\lambda$ is regular, then the vector $v=\sum_{\lambda\in\Lambda_\xi}\sqrt{a_\lambda}v^\lambda$ gives a critical point $[v]$ of $||\mu||^2$ mapped by $\mu$ to $\xi$.

Another way to prove the proposition is to use Corollary \ref{coro_centraliser}. The centralizer of a regular $\xi$ is equal to the torus, $T\subset K$ and its semisimple part consists of the identity element, $e\in K$. Then the momentum map $\mu_{K'_\xi}$ is trivial, i.e. assigns zero to each element of $Z_\xi$.
\end{proof}

The case when $\Lambda_\xi$ is not root distinct, which can happen only if $\xi$ is not regular, is more difficult to deal with. We next give a solution to this problem by finding roots and weights of the representation of $K'_\xi$ on $V_\xi$. As in the previous sections, it is useful to distinguish two types of simple roots - the ones that are orthogonal to $\xi$, $\Pi^{(0)}_\xi$, and the remaining part, $\Pi'_{\xi}$. Using the decomposition of $\xi$ in the basis of fundamental weights, $\xi=\sum_{\alpha\in\Pi}\xi_\alpha\omega_\alpha$, one can characterise the partition of $\Pi$ in the following way
\begin{gather*}
 \Pi^{(0)}_\xi=\{\alpha\in\Pi:\ \xi_\alpha=0\}, \\
 \Pi'_{\xi}=\Pi- \Pi^{(0)}_\xi=\{\alpha\in\Pi:\ \xi_\alpha>0\}.
\end{gather*}

\begin{lemma}\label{lemma_centr_rep}
Suppose that $\Lambda_\xi,\ \xi\in\mc{B}$, is not root-distinct. Moreover, denote the centralizer of $\xi$ in $K$ by $K'_\xi$. Then $ \Pi^{(0)}_{\xi}$ are the simple roots $K'_\xi$ and the weights of the representation of $K'_\xi$ on $V_\xi$ are given by the orthogonal projection of $\Lambda_\xi$ on $\mathrm{Span}_\ZZ\left(\Pi^{(0)}_{\xi}\right)$.
\end{lemma}
\begin{proof}
First, we will show that the simple roots of $K'_\xi$ are given by $\Pi^{(0)}_{\xi}$. The fact that the roots from $ \Pi^{(0)}_{\xi}$ are perpendicular to $\xi$ implies that 
\begin{equation*}
ad_{h_\alpha}(h_\xi)=[h_\alpha,h_\xi]=(\alpha|\xi)h_\xi=0.
\end{equation*}
Using the Baker-Campbell-Hausdorff formula 
\begin{equation*}
e^{ih_\alpha}h_\xi e^{-ih_\alpha}=h_\xi+i[h_\alpha,h_\xi],
\end{equation*}
one can see that $h_\alpha$ belongs to the centralizer of $h_\xi$ if and only if $[h_\alpha,h_\xi]=0$, that is, if and only if $(\alpha|\xi)=0$. All such roots belong to the span of $\Pi^{(0)}_{\xi}$. Let us next find the weights of $K'_\xi$. Note that the Lie algebra of the maximal torus, $T'_\xi\subset K'_\xi$ is of the form
\begin{equation*}
\mk{t}'_\xi=\mathrm{Span}_\RR\left\{ih_\alpha,\ \alpha\in\Pi_0^{(\xi)}\right\},
\end{equation*}
hence $T'_\xi$ a subgroup of the maximal torus $T\subset K$. Therefore, the fixed points for the action of $T$ on $v_\xi$ are also fixed by the action of $T'_\xi$. This in turn means that the weight spaces of $T$ are contained in the weight spaces of $T'_\xi$. Let us now calculate the momentum image of each weight space $V^\nu,\ \nu\in\Lambda_\xi$, which is, by Lemma \ref{Lemma Wildberger}b, equal to the corresponding weight $\nu$. For $\eta\in\mk{t}'_\xi$, we have
\begin{equation*}
\mu_{K'_\xi}([v^{\nu}])(ih_\eta)=\frac{1}{i}\frac{\scal{ih_\eta.v^\nu}{v^\nu}}{\scal{v^\nu}{v^\nu}}=\frac{\scal{(\eta|\nu)v^\nu}{v^\nu}}{\scal{v^\nu}{v^\nu}}=(\eta|\nu).
\end{equation*}
In other words, all components of each weight are simply the components of the weight projected orthogonally on $\mathrm{Span}_\ZZ\left(\Pi^{(0)}_{\xi}\right)$.
\end{proof}

\begin{rem}
In order to decide, whether $0$ is in the image of $\mu_{K'_\xi}:Z_\xi\lw({\mk k}'_\xi)^*$, one can use Lemma \ref{lemma_centr_rep} to find roots and the weight lattice of representation $K'_\xi$. If the result does not match any of the known representations, where the only invariant polynomials are the constants, then $\xi$ is not an intermediate critical value.
\end{rem}

\subsection{Example: the representation of $GL_2^{\times N}$ on $(\CC^2)^{\otimes N}$}\label{Example qubits}

The critical points of $||\mu||^2$ for this example have been computed by the second author and A. Sawicki, \cite{Tomek-Adam-LCubitCrit}, in the context of quantum entanglement, using a different method. We present here a calculation based on the method outlined above.

Denote by $\pm\delta$ the weights of the natural representation of $GL_2$. Then the weight lattice is of the form
$$
\Lambda(V)=\{\sigma_{i_1}\delta\oplus\sigma_{i_2}\delta\oplus\dots\oplus\sigma_{i_N}\delta\ :\ \sigma_{i_k}=\pm1\},
$$
i.e. is isomorphic to the set of vertices of the $N$-dimensional hypercube. The Weyl group is a group of simple reflections with respect to the plains crossing the centre of the hypercube which are perpendicular to some of the hypercube's edges. Moreover, the highest weight is $\delta\oplus\delta\oplus\dots\oplus\delta$ and all weights belong to the same Weyl orbit. Any $\lambda\in\Lambda(V)$ regarded as a vector in $\mk{t}$, can be written as a combination of simple roots in the following way
$$
\lambda=\sigma^\lambda_1\delta\oplus\dots\oplus\sigma^\lambda_N\delta=\frac{1}{2}\left(\sigma^\lambda_1\alpha_1+\dots+\sigma^\lambda_N\alpha_N\right),
$$
where $\alpha_i$ is a simple root corresponding to the $i$-th component of $GL_2^{\times N}$. Similarly, one can express $\lambda$ in terms of the fundamental weights
$$
\lambda=\sigma^\lambda_1\omega_1+\dots+\sigma^\omega_N\alpha_N,
$$
where $(\omega_i|\alpha_j)=\delta_{ij}$.
Assume now that $\xi\in\mk{t}$ belongs to a wall of the positive Weyl chamber and that $\Lambda_\xi$ is nonempty. Such a $\xi$ can be written in the basis of simple roots as 
$$
\xi=\sum_{i\in\mc{I}}c_i\alpha_i,\ c_i>0,
$$
where $\mc{I}$ is a proper subset of $\{1,2,\dots,N\}$. One can now see by a straightforward calculation that 
$$
\Pi^{(0)}_\xi=\{\alpha_k:\ k\in\{1,\dots,N\}-\mc{I}\}.
$$
Using Lemma \ref{lemma_centr_rep}, one can easily find weights of the representation of $K'_\xi$ on $V_\xi$, namely if $\lambda$ belongs to $\Lambda_\xi$, then $\lambda'=\sum_{i\in\{1,\dots,N\}-\mc{I}}\frac{1}{2}\sigma^\lambda_i\alpha_i$ is a weight of $K'_\xi$. From now on, let us denote by $N_0$ the number of roots in $\Pi^{(0)}_\xi$, i.e. the number of simple roots orthogonal to $\xi$.
\begin{prop}
For $\xi$ from the boundary of the Weyl chamber, the image of the moment map $\mu_{K'_\xi}$ contains 0 if and only if $N_0>1$.
\end{prop}
\begin{proof}
If $N_0>1$, the set weights $\Lambda'_\xi$ always contains a root-distinct pair of weights of the opposite sign. Then zero is contained in the line segment between those two weights and by Lemma \ref{Lemma Wildberger} the whole segment belongs to the image of $\mu_{K'_\xi}$. On the other hand, if $N_0=1$, then $\Lambda'_\xi$ contains only two non root-distinct weights, therefore zero is not contained in the momentum image.
\end{proof}
Thus, in the case when $N_0>1$, $\xi$ is an intermediate critical value. It is also possible to construct an exemplary critical point from the preimage of $\xi$.
\begin{prop}
Let $\xi\in\mk{t}$ be an intermediate critical value from the boundary of the Weyl chamber. Consider the representation of $K$ as a tensor product of the representation of $K'_\xi$ acting on $V^{(0)}_\xi=\left(\CC^2\right)^{\otimes N_0}$ and the remaining part of $K$, $K^{(1)}$, acting on $V^{(1)}_\xi=(\CC^2)^{\otimes(N-N_0)}$. Then point of the following form is a critical point mapped by $\mu_K$ to $\xi$
\begin{equation}
[v]=\left[w\otimes v^{\lambda}+w^\perp\otimes v^{-\lambda}\right],
\label{preimage_point}
\end{equation}
where $v^\lambda, v^{-\lambda}$ is a pair of weight vectors for the action of $K'_\xi$ on $V^{(0)}_\xi$ and $w$ is such a vector that $\mu_{K^{(1)}}([w])=\rm{pr}_{\Pi'_{\xi}}(\xi)=\xi$. Moreover, $w^\perp$ is a vector orthogonal to $w$ that belongs to the orbit of a maximal torus of $K^{(1)}$ through $w$.
\end{prop}
\begin{proof}
Showing that for $v$ given by (\ref{preimage_point}) $\mu_K([v])=\xi$, is rather a straightforward task, which can be done by calculating all components of $\mu$ along $\alpha\in\Pi$. Similarly, showing that $T_\xi$ is in the stabilizer of $[v]$ is also a matter of straightforward calculation. What remains to be shown is the existence of $w^\perp$. Let us first write $\xi$ as a convex combination of weights from $\Lambda^{(1)}_\xi=\rm{pr}_{\Pi'_{\xi}}\left(\Lambda_\xi\right)$, the weight lattice for the representation $V^{(1)}_\xi$. Following Corollary \ref{coro_minimal_weights}, it is possible to choose $\rm{dim}\left(\mk{t}^{(1)}\right)$ linearly independent weights, where $\mk{t}^{(1)}$ is the Cartan subalgebra of $K^{(1)}$, i.e.
$$
\xi=\sum_{\nu\in\Lambda^{(1)}_\xi}b_\nu\nu,\ \ \#\{\nu:\ b_\nu\neq0\}\leq\rm{dim}\left(\mk{t}^{(1)}\right).
$$
Because $\Lambda^{(1)}_\xi$ is root-distinct, vector $w$ can be written as
$$
w=\sum_{\nu\in\Lambda^{(1)}_\xi}\sqrt{b_\nu}v^\nu.
$$
For any $\tau\in\mk{t}^{(1)}$ we also have
$$
e^\tau w=\sum_{\nu\in\Lambda^{(1)}_\xi}\sqrt{b_\nu}e^\tau v^\nu=\sum_{\nu\in\Lambda^{(1)}_\xi}\sqrt{b_\nu}e^{i(\tau|\nu)}v^\nu.
$$
Therefore, the condition $\langle w,e^\tau w\rangle=0$ now reads
\begin{equation}
\sum_{\nu\in\Lambda^{(1)}_\xi}b_\nu e^{i(\tau|\nu)}=0
\label{eq:perp_cond}
\end{equation}
Equation (\ref{eq:perp_cond}) can be seen as a problem of requiring a set of complex numbers $\{b_{\nu}e^{i(\tau|\nu)}\}$ to add up to zero. In other words, we want to form a polygon of $\#\{\nu:\ b_\nu\neq0\}$ sites with lengths $\{b_{\nu}\}$ and angles between them given by $\{(\tau|\nu)\}$. The conditions sufficient to construct such polygon are:
\begin{enumerate}
\item the length of the longest side is smaller than the sum of the lengths of the remaining sides (polygon inequality),
\item  there are at least 2 sides and at most $\rm{dim}(\mk{t}^{(1)})$ sides.
\end{enumerate}
The second condition means that the number of polygon angles cannot be greater than the dimension of the vector space that $\tau$ belongs to, as weights from $\Lambda'_\xi$ are chosen to be linearly independent and numbers $(\tau|\nu)$ can be viewed as components of vector $\tau$. This condition is always satisfied. The polygon inequality is also satisfied thanks to the fact that all weights from $\Lambda_\xi$, and therefore from $\Lambda'_\xi$, belong to the same Weyl orbit.
\end{proof}

\begin{rem}
The above reasoning can be extended to a more general case when the root system can be written as an orthogonal direct sum of $\Pi^{(0)}_\xi$ and $\Pi'_\xi$ and all weights from $\Lambda_\xi$ are congruent to each other under the action of the Weyl group.
\end{rem}

\section{Secant varieties and degrees of invariants}\label{Sect SecVar}

Recall that $\XX\subset \PP$ denotes the unique closed $G$-orbit in the projective space of and irreducible module $V$ associated to a dominant weight $\lambda$. The rank function on $\PP$ with respect to $\XX$ is defined as
$$
{\rm rk}_\XX:\PP\lw\NN \;,\quad {\rm rk}_\XX[v]=\min\{r\in\NN:v=x_1+...+x_r, [x_j]\in\XX\} \;.
$$
The rank subsets of $\PP$ and the secant varieties of $\XX$ are defined respectively as
$$
\XX_r = \{[v]\in\PP:{\rm rk}_\XX[v]=r\}\quad,\quad \Sigma_r=\ol{\cup_{s\leq r}\XX_s} \;.
$$
The border rank on $\PP$ with respect to $\XX$ is defined as
$$
\ul{\rm rk}_\XX :\PP\le\NN \;,\quad \ul{\rm rk}_\XX[v]=\min\{r\in\NN:[v]\in\Sigma_r \;.
$$
There is a unique $r_m\geq1$, for which $\XX_{r_m}$ is open in $\PP$ and this is the smallest rank for which $\Sigma_r=\PP$. The rank function, and consequently $\XX_r$ and $\Sigma_r$, are $G$-invariant. We have containments of varieties $\XX\subset\Sigma_2\subset...\subset\Sigma_{r_m}=\PP$ and there is a corresponding chain of $G$-stable ideals $I(\XX)\supset I(\XX_2)\supset...\supset 0$. The ideal of $\XX$ is generated in degree 2, by a suitable generalization of Pl\"ucker's relations, due to Kostant, see e.g. \cite{Landsberg-2012-book}.

\begin{theorem}{\rm (Landsberg and Manivel, \cite{Lands-Mani-2004-IdealsSecVar})}
 
The first nonzero homogeneous component of the ideal $I(\Sigma_r)$ is in degree $r+1$ and is given, for $r\geq 2$, by the $(r-1)$-st prolongation of the generating space $I_2(\XX)$ of the ideal of $\XX$, i.e.
$$
I_r(\Sigma_r)=0 \quad , \quad I_{r+1}(\Sigma_r)=S^{r+1}V^* \cap (I_2(\XX)\otimes S^{r-1}V^*) \;.
$$
\end{theorem}

Consider the ring $\CC[V]^G$ of invariant polynomials on the representation $V$. Let $J\subset\CC[V]^G$ be the ideal in the invariant ring vanishing at $0$. Let $\PP_{us}\subset\PP$ be the zero-locus of $J$, also known as the unstable locus, or the nullcone. The complement $\PP_{ss}=\PP\setminus\PP_{us}$ is called the semistable locus. We are interested in the relations between secant varieties and their ideals on one side, and the nullcone and the ring of invariants, on the other. Recall that the incidence with the nullcone for a projective variety can be tested via the momentum map: $\Sigma_r\subset\PP_{us}$ if and only if $0\notin \mu(\Sigma_r)$.

\begin{defin} If $\CC[V_\lambda]^G\ne\CC$, the rank of semistability of $V_\lambda$ is defined as
$$
r_{ss} = \min\{r\in\NN:0\in\mu(\Sigma_r)\} \;.
$$
\end{defin}

The closed $G$-orbit $\XX\subset\PP$ belongs to $\PP_{us}$ as long as the representation is nontrivial; thus $r_{ss}\geq2$. The following proposition is an interpretation of a result of Zak, \cite{Zak-Book}, Ch. III.

\begin{prop}

If $V_\lambda\ncong V_\lambda^*$, then $\Sigma_2\in\PP_{us}$. If $\CC[V_\lambda]^G\ne\CC$, then $r_{ss}=2$ if and only if $V_\lambda\cong V_\lambda^*$.
\end{prop}

We know that $0\notin\mu(\Sigma_1)=\mu(\XX)=K\lambda$, so $r_{ss}\geq 2$. The above results have the following direct consequence.

\begin{prop}\label{Prop Sec and d-Inv}
Let $d_1$ be the minimal positive degree of a polynomial in $\CC[V]^G$. Let $d_1'$ be the minimal positive degree of an invariant polynomial nonvanishing on the first semistable secant variety $\Sigma_{r_{ss}}$. Then
$$
r_{ss}\leq d_1\leq d_1' \;.
$$
\end{prop}

There is a class of homogeneous projective varieties $\XX\subset\PP(V)$, called rank-semi-continuous, or briefly rs-continuous varieties, for which the above relation between $r_{ss}$ and $d_1$ proves to be stronger. They characterized by the requirement that rank and border rank should coincide, i.e. $\XX$ is called {\it rs-continuous} if ${\rm rk}_\XX=\ul{\rm rk}_\XX$. The rs-continuous varieties have been classified by one of the authors and and A. V. Petukhov, \cite{Petukh-Tsan}. The list is given in the table below, along with the linear automorphism groups of the varieties. We also give the degree $d_1$, which we have obtained from Kac's table \cite{Kac-nilp-orb}. (All rs-continuous varieties turn out to have polynomial invariant rings.) We put $d_1=0$ if $\CC[V]^G=\CC$.\\

\begin{center}
{\bf rs-continuous varieties}\\
\vspace{0.5cm}
\begin{tabular}{|c|c|c|c|c|}
\hline
Notation for $\XX$ & Ambient~$\PP(V)$ & Group $G$& Max rk$_\XX$ & $d_1$ \\
\hline
$\PP(\CC^n)$ &$\PP(\CC^n)$& $SL_n$ & $1$ & 0\\
\hline
${\rm Ver}_2(\PP(\CC^n))$&$\PP({\rm S}^2\CC^n)$ & $SL_n$ & $n$ & $n$ \\
\hline
Gr$_2(\CC^n)$&$\PP(\Lambda^2\CC^n)$ & $SL_n$ & $\lfloor\frac{n}{2}\rfloor$ & 0 for odd $n$; \\ &&&& $n/2$ for even $n$\\
\hline
Fl$(1,n-1; \CC^n)$&$\PP(\mk{sl}_n)$ & $SL_n$ & $n$ & 2 \\
\hline
Q$^{n-2}$&$\PP(\CC^n)$ & $SO_n$ & $2$ & 2 \\
\hline
S$^{10}$&$\PP(\CC^{16})$ & $Spin_{10}$ & $2$ & 0 \\
\hline
Gr$_{\omega}(2,\CC^{2n})$&$\PP(\Lambda_0^2\CC^{2n})$ & $Sp_{2n}$ & n & 2\\
\hline
E$^{16}$&$\PP(\CC^{27})$ & $E_6$ & $3$ & 3 \\
\hline
F$^{15}$&$\PP(\CC^{26})$ & $F_4$ & $3$ & 2 \\
\hline
${\rm Segre}(\PP(\CC^m)\times\PP(\CC^n))$&$\PP(\CC^m\otimes\CC^n)$ & $SL_m\times SL_n$ & $\min\{m,n\}$ & 0 for $m\ne n$; \\ &&&& $m$ for $m=n$ \\
\hline
\end{tabular}
\end{center}

The nullcones are also known in all cases, see e.g. in Zak, \cite{Zak-Book}, Ch. 3. We obtain the following.

\begin{prop}
Suppose that $\XX\subset\PP$ is rs-continuous and $G$ is the linear automorphism group of $\XX$. Then
$$
r_{ss}=d_1=d_1' \;.
$$
\end{prop}

\begin{ex}
Consider $G=SL_n$ acting on $V=S^k\CC^n$ with $k,n\geq2$. The associated homogeneous projective variety is the Veronese variety $\XX={\rm Ver}_k(\PP^{n-1})\subset\PP(V)$. Then $\mu(\PP(V))\cap{\mk t}^*$ is an $n$-simplex centered at $0$. It is not hard to see that, for $1\leq r\leq n$, $G$ has an open orbit in $\Sigma_r$ and the momentum image $\mu(\Sigma_r)\cap{\mk t}^*$ is the $r$-skeleton of the simplex. Thus $\Sigma_r\subset\PP_{us}$ for $r<n$ and $r_{ss}=n$, so that $d\geq n$. The variety $\XX$ is rs-continuous if an only if $k=2$; then the determinant of symmetric matrices is an invariant of degree $n$. The case $n=3$ shows that, for large $k$, invariants in degree $n$ may or may not occur.
\end{ex}

\section{Appendix: reducible unstable representations}

Here we present a proof of Proposition \ref{Prop UnstableMultFree}, based on Wildberger's method. Recall that the proposition states that the only reducible multiplicity-free representations $V$ of a semisimple compact group $K$ for which $0 \notin \mu(\PP(V))$ are the $SU_{2n+1}$-representation on $(\CC^{2n+1})^* \oplus \Lambda^2\CC^{2n+1}$ and its dual.

\begin{proof}
Up to duality the only candidates for multiplicity-free reducible representations $V$ with $0 \notin \mu(\PP(V))$ are the following:
\begin{itemize}
\item the $SU_{2n+1}$-representations $\CC^{2n+1} \oplus \Lambda^2\CC^{2n+1}$ and $(\CC^{2n+1})^* \oplus \Lambda^2\CC^{2n+1}$;
\item the $SU_{2} \times SU_5$-representations $(\CC^{2}\otimes\CC^5) \oplus (\CC^2\otimes\Lambda^2\CC^{5})$ and $(\CC^{2}\otimes (\CC^5)^*) \oplus (\CC^2\otimes\Lambda^2\CC^{5})$;
\item the $SU_{2} \times SU_7$-representations $(\CC^{2}\otimes\CC^7) \oplus (\CC^2\otimes\Lambda^2\CC^{7})$ and $(\CC^{2}\otimes (\CC^7)^*) \oplus (\CC^2\otimes\Lambda^2\CC^{7})$.
\end{itemize}

We check explicitly in each case if $0$ could be written as a convex combination of root-distinct weights. 

Case $V = \CC^{2n+1} \oplus \Lambda^2\CC^{2n+1}$. We denote the weights of $\CC^{2n+1}$ as a $GL_{2n+1}$-representation by $\varepsilon_j$, $1 \leq j \leq 2n+1$. Then the weights of $\Lambda^2\CC^{2n+1}$ are of the form $\varepsilon_i + \varepsilon_j$ for $1\leq i < j \leq 2n+1$. Notice that $(\varepsilon_i + \varepsilon_j) -(\varepsilon_k + \varepsilon_l)$ is a root if and only if $\#(\{i,j\} \cap \{k,l\}) = 1$. Moreover, $(\varepsilon_i + \varepsilon_j) - \varepsilon_k$ is never a root. Therefore, we could express $0$ in the following way:
$$
0 = \frac{1}{n+1} (\varepsilon_1 + \varepsilon_2 + \dots + \varepsilon_{2n+1}) = \frac{1}{n+1}(\varepsilon_{1} + \varepsilon_{2}) + \frac{1}{n+1}(\varepsilon_{3} + \varepsilon_{4}) +\dots + \frac{1}{n+1}(\varepsilon_{2n-1} + \varepsilon_{2n}) + \frac{1}{n+1}\varepsilon_{2n+1}.
$$
Thus $0 \in \mu(\PP(V))$.\\

Case $V = (\CC^{2}\otimes\CC^5) \oplus (\CC^2\otimes\Lambda^2\CC^{5})$. We denote the weights of $\CC^{2}$ as a $GL_2$-representation by $\delta_1$ and $\delta_2$ and the weights of $\CC^{5}$ as a $GL_5$-representation by $\varepsilon_j$ for $1 \leq j \leq 5$. Then
$$
\Lambda(V) = \{\delta_i \oplus \varepsilon_j : i=1,2; 1\leq j \leq 5 \} \cup \{\delta_i \oplus (\varepsilon_j + \varepsilon_k) : i=1,2; 1\leq j<k \leq 5 \}.
$$
Moreover,
\begin{align*}
0 = &\frac{1}{3}((\delta_1 + \delta_2) \oplus(\varepsilon_1 + \dots + \varepsilon_5)) = \\
&\frac{1}{6}(\delta_1 \oplus(\varepsilon_1 + \varepsilon_2) + \delta_1 \oplus (\varepsilon_3 + \varepsilon_4) + \delta_1 \oplus \varepsilon_5 + \delta_2 \oplus(\varepsilon_1 + \varepsilon_3) + \delta_2 \oplus (\varepsilon_2 + \varepsilon_5) + \delta_2 \oplus \varepsilon_4.
\end{align*}
Thus $0 \in \mu(\PP(V))$.\\

Case $V = (\CC^{2}\otimes (\CC^5)^*) \oplus (\CC^2\otimes\Lambda^2\CC^{5})$.  Then
$$
\Lambda(V) = \{\delta_i \oplus (\varepsilon_1 + \dots + \hat{\varepsilon_{j}} + \dots +\varepsilon_5) : i = 1,2; 1 \leq j \leq 5\} \cup\{\delta_i \oplus (\varepsilon_j + \varepsilon_k) : i=1,2; 1\leq j<k \leq 5 \}.
$$
Moreover,
\begin{align*}
0 = & \frac{1}{2} (\delta_1 + \delta_2) \oplus (\varepsilon_1 + \dots + \varepsilon_5) = \\
&\frac{1}{4}(\delta_1\oplus(\varepsilon_1 + \varepsilon_2 + \varepsilon_3 + \varepsilon_4) + \delta_1\oplus(\varepsilon_4 + \varepsilon_5) + \delta_2\oplus(\varepsilon_1 + \varepsilon_2) + \delta_2\oplus(\varepsilon_3 + \varepsilon_5)).
\end{align*}
Thus again $0 \in \mu(\PP(V))$.\\

Case $V = (\CC^{2}\otimes\CC^7) \oplus (\CC^2\otimes\Lambda^2\CC^{7})$. Then
$$
\Lambda(V) = \{\delta_i \oplus \varepsilon_j : i=1,2; 1\leq j \leq 7 \} \cup \{\delta_i \oplus (\varepsilon_j + \varepsilon_k) : i=1,2; 1\leq j<k \leq 7 \}.
$$
Moreover,
\begin{align*}
0 = &(\frac{1}{2}(\delta_1 + \delta_2)) \oplus(\frac{1}{4}(\varepsilon_1 + \dots + \varepsilon_5)) = \\
&\frac{1}{4}(\delta_1 \oplus(\varepsilon_1 + \varepsilon_2) + \delta_2 \oplus (\varepsilon_3 + \varepsilon_4) + \delta_1 \oplus (\varepsilon_5 + \varepsilon_6) + \delta_2 \oplus \varepsilon_7).
\end{align*}
Thus $0 \in \mu(\PP(V))$.\\

Case $V = (\CC^{2}\otimes (\CC^7)^*) \oplus (\CC^2\otimes\Lambda^2\CC^{7})$.  Then
$$
\Lambda(V) = \{\delta_i \oplus (\varepsilon_1 + \dots + \hat{\varepsilon_{j}} + \dots +\varepsilon_7) : i = 1,2; 1 \leq j \leq 7\} \cup\{\delta_i \oplus (\varepsilon_j + \varepsilon_k) : i=1,2; 1\leq j<k \leq 7 \}.
$$
Moreover,
\begin{align*}
0 = & \frac{1}{2} (\delta_1 + \delta_2) \oplus (\varepsilon_1 + \dots + \varepsilon_7) = \frac{1}{4}(\delta_2\oplus(\varepsilon_1 + \varepsilon_2 + \dots + \varepsilon_6) + \delta_1\oplus(\varepsilon_1 + \varepsilon_7)) + \\
&\frac{1}{8}(\delta_1\oplus(\varepsilon_2 + \varepsilon_3 + \dots + \varepsilon_7) + \delta_1\oplus(\varepsilon_2 + \varepsilon_3) + \delta_2\oplus(\varepsilon_4 + \varepsilon_5) + \delta_2\oplus(\varepsilon_6 + \varepsilon_7)).
\end{align*}
Thus again $0 \in \mu(\PP(V))$.\\

It remains to show that, when $V = (\CC^{2n+1})^* \oplus \Lambda^2\CC^{2n+1}$, then $0 \notin \mu(\PP(V)$. This representation is such that $G=K^\CC=SL_{2n+1}$ acts spherically on the projective space $\PP=\PP(V)$. The description of spherical representations by Knop, \cite{Knop-MultFreeSpa}, includes the information $\CC[V]^G=0$, which is equivalent to $0\notin\mu(\PP)$. 

\end{proof}

\bibliographystyle{plain}

\small{

}

\vspace{0.4cm}

\author{\noindent Elitza Hristova \\ Fakult\"at f\"ur Mathematik, Ruhr-Universit\"at Bochum,\\ Universit\"atstra{\ss}e 150, 44780 Bochum, Germany. \\ elitza.hristova@rub.de}\\

\author{\noindent Tomasz Maci\c{a}\.{z}ek \\ Center for Theoretical Physics, Polish Academy of
Sciences \\ Al. Lotnik\'ow 32/46, 02-668 Warszawa, Poland. \\ maciazek@cft.edu.pl}\\

\author{\noindent Valdemar V. Tsanov \\ Mathematisches Institut, Universit\"at G\"ottingen,\\ Bunsenstra{\ss}e 3-5, 37073 G\"ottingen, Germany. \\ valdemar.tsanov@mathematik.uni-goettingen.de}\\

\end{document}